\documentclass[12pt]{amsart}

\usepackage[usenames,dvipsnames]{xcolor}
\usepackage{fancyhdr}
\usepackage{amsmath,amsthm,amsfonts,graphicx,amssymb,amscd}
\usepackage[all,cmtip]{xy}
\usepackage{graphicx, overpic, float}
\usepackage[mathscr]{euscript}
\usepackage{hyperref}
\usepackage{enumerate}

\newtheorem{thm}{Theorem}[section]

\newtheorem{theorem}[thm]{Theorem}

\newtheorem{lem}[thm]{Lemma}
\newtheorem{lemma}[thm]{Lemma}
\newtheorem{prop}[thm]{Proposition}
\newtheorem{proposition}[thm]{Proposition}
\theoremstyle{definition}
\newtheorem{defn}[thm]{Definition}
\newtheorem{definition}[thm]{Definition}

\newcommand{\ML}{\mathscr{ML}}

\newcommand{\GL}{\mathcal{GL}}

\newcommand{\RR}{\mathcal{R}}
\newcommand{\Gr}{\operatorname{Gr}}
\newcommand{\gr}{\operatorname{gr}}

\newcommand{\T}{\mathscr{T}}
\newcommand{\TT}{\mathcal{T}}
\newcommand{\TTT}{\mathsf{T}}
\renewcommand{\P}{\mathscr{P}}
\newcommand{\PP}{\mathcal{P}}
\renewcommand{\AA}{\mathcal{A}}
\newcommand{\AAA}{\mathsf{A}}
\newcommand{\FF}{\mathcal{F}}
\newcommand{\LL}{\mathcal{L}}

\newcommand{\NNN}{\mathsf{N}}
\newcommand{\Z}{\mathbb{Z}}
\newcommand{\R}{\mathbb{R}}
\newcommand{\RRR}{\mathsf{R}}
\newcommand{\C}{\mathbb{C}}
\renewcommand{\H}{\mathbb{H}} 
\newcommand{\h}{\mathbb{H}} 
\newcommand{\RS}{{\hat{\mathbb{C}}}}

\newcommand{\rs}{\hat{\mathbb{C}}}

\newcommand{\M}{\mathscr{M}}

\newcommand{\hol}{{\rm hol}}

\newcommand{\PSL}{\operatorname{PSL_2(\mathbb{C})}}

\newcommand{\PSLr}{\operatorname{PSL(2, \mathbb{R})}}

\newcommand{\Conv}{\operatorname{Conv}}

\newcommand{\length}{\operatorname{length}}

\newcommand{\bdr}{\partial}

\newcommand{\del}{\delta}
\newcommand{\ep}{\epsilon}
\newcommand{\lam}{\lambda}
\newcommand{\kap}{\kappa}
\newcommand{\al}{\alpha}
\newcommand{\gam}{\gamma}

\newcommand{\col}{\colon}
\newcommand{\ci}{\circ}
\newcommand{\sub}{\subset}

\usepackage{fancybox}
\newcommand{\minus}{\setminus}

\newcommand{\til}{\tilde}

\newcommand{\Label}[1]{\label{#1}\textcolor{green}{\tiny #1} }
\renewcommand{\Label}[1]{\label{#1}}

\newcommand{\Proof}{{\noindent \it Proof. }}
\newcommand{\Qed}[1]{\nopagebreak[4]{\tiny \hfill
\fbox{\ref{#1}} \linebreak }\pagebreak[2]}

\newcommand{\bd}{\partial}

\renewcommand{\emph}[1]{{\it #1}}% to make emph{} not underline.

\usepackage{color}

\title[Holonomy map fibers of $\C P^1$-structures in Moduli space]{Holonomy map fibers of $\C P^1$-structures\\  in Moduli space}

\author{Shinpei Baba}
\address{Universit\"at Heidelberg}
\email{shinpei@mathi.uni-heidelberg.de}

\author{Subhojoy Gupta}
\address{Center for Quantum Geometry of Moduli Spaces}
\email{sgupta@qgm.au.dk}

\date{\today}

\begin{document}
\maketitle
\begin{abstract}
Let $S$ be a closed oriented surface of genus $g\geq 2$.  Fix an arbitrary non-elementary representation $\rho\col\pi_1(S)\to {\rm SL}_2(\C)$  and consider all marked (complex) projective structures on $S$ with holonomy $\rho$. 
 We show that  their underlying conformal structures are dense in the moduli space of $S$. 
\end{abstract}

\section{Introduction}

Let $F$ be an oriented connected surface. 
 A {\it (complex) projective structure} on $F$ is a maximal atlas modeled on the Riemann sphere $\RS$ with transition maps in $\PSL$.
 Then each projective structure enjoys a conformal structure induced by $\RS$. 
 Letting $\til{F}$ be the universal cover of $F$, a projective structure is equivalently given as a pair of  an immersion $f:\widetilde{F}\to \RS$ (\textit{developing map}) and a homomorphism $\rho:\pi_1(F)\to \PSL$  (\textit{holonomy}) such that $f$ is $\rho$-equivariant. 

Let $S$ be a closed oriented surface of genus $g\geq 2$. 
Let $\mathscr{P}$ be  the space of all marked projective structures on $S$, which is diffeomorphic to $\C^{6g-6}$.
Then let $\chi$ be the component of the  $\PSL$-character variety of $\pi_1(S)$ containing the representations $\pi_1(S) \to \PSL$ lifting to  $\pi_1(S) \to {\rm SL}_2(\C)$ (c.f. \cite{Goldman-88t}).
The {\it holonomy map}
\begin{equation*}
{\rm hol}:\mathscr{P}\to \chi
 \end{equation*}
is the natural forgetful map taking projective structures to their holonomy representations.\\

A subgroup of $\PSL$ is {\it elementary} if its action preserves a single point or two points  in $\H^3 \cup \RS$.
Then the image of $\hol$ consists of  all representations in $\chi$ whose images are {\it not} elementary (\cite{Gallo-Kapovich-Marden}).
Note that ${\rm Im}(\hol)$ contains many nondiscrete representations.
A basic question  is to understand the {\it holonomy fiber} $\hol^{-1}(\rho) =: \mathscr{P}_\rho$ in $\mathscr{P}$ (\cite{Goldman-thesis, Kapovich-95, Gallo-Kapovich-Marden, Dumas-08}). 
In particular \cite[p 274]{Hubbard-81} asked what  $\P_\rho$ looks like in the Teichm\"{u}ller space $\T$ by the projection, $p \col \P \to \T$, from projective structures
to their underlying conformal structures.
Note that  $\P$ is a complex vector bundle over $\T$ by $p$.

In this paper, we consider its further projection into the moduli space $\mathscr{M}$ of Riemann surfaces,  given by  the action of the mapping class group of $S$ on $\T$, 
\begin{equation*} 
\mathscr{P} \xrightarrow{p} \mathscr{T} \xrightarrow{\pi} \mathscr{M}.
\end{equation*}
\vspace{.1in}
We shall prove

\begin{thm}\label{thm:thm1}
Let $\rho:\pi_1(S)\to \PSL$ be any non-elementary homomorphism that lifts to $\pi_1(S) \to {\rm SL}(2, \C)$.
 Then the holonomy fiber $\mathscr{P}_\rho$ projects onto a dense subset of the moduli space $\M$. 
\end{thm}
\vspace{.1in}

 In contrast $\P_\rho$ is a discrete subset of $\mathscr{P}\, (\cong \C^{6g-6})$  since $\hol$ is a local homeomorphism (\cite{Hejhal-75}).
 Furthermore $\mathscr{P}_\rho$ injectively  maps onto a \textit{discrete} subset of  $\mathscr{T}$ by $p$.
Indeed, for any compact set $K \subset \mathscr{T}$, the restriction of  $\hol$ to $p^{-1}(K)$ is proper (see Theorem 3.2 of \cite{Tanigawa99}). 
Therefore $\mathscr{P}_\rho \cap p^{-1}(K)$ is a finite set in $\P$.
Then its projection into $\T$ is $p(\mathscr{P}_\rho)\cap K$ and hence it is discrete. 
For $X \in \T$,  let $\P_X$ be the fiber $p^{-1}(X)$ of $p$ over $X$, which is a complex vector space of dimension $3g-3$. 
Then $\hol$ properly embeds $\P_X$ into $\chi$ (\cite{Gallo-Kapovich-Marden}), which implies the injectivity of $p | \P_\rho$.

The proof of Theorem \ref{thm:thm1} utilizes the  operation of projective structures, called \textit{grafting}: given a hyperbolic surface $\tau\in \mathscr{T}$ and a {measured geodesic lamination} $L$, grafting yields a new projective surface $Gr_L (\tau)$. If $L$ is a multiloop, it inserts projective annuli along the components of $L$, where the heights of the annulus are given by weights. 
 This gives a geometric parameterization $\mathcal{P} \cong \mathscr{T}\times \mathscr{ML}$, called  {\it Thurston coordinates} (see \cite{Kamishima-Tan-92, Kullkani-Pinkall-94}), where $\ML$ is the set of all measured laminations on $S$.
In particular, grafting along multiloops weighted by $2\pi$-multiples preserves the Fuchsian holonomy (see \cite{Goldman-87}).
Moreover this $``2\pi"$-grafting generalizes to grafting of general projective structures along ``admissible" loops (see \S2.2), which preserves holonomy as well.\\

When grafting a hyperbolic surface $\tau$, 
by scaling the transversal measure on the lamination $L$, we obtain a ray of projective structures,  $Gr_{tL} (\tau)$ ($t\geq 0$), in $\mathscr{P}$. 
Then  it descends to a ray $\gr_{t L}$ in $\T$, and then to a ray in $\M$. 

Then \cite{Gupta14} and \cite{Gupta14-2} yield the strong asymptoticity of these grafting rays to Teichm\"{u}ller geodesic rays, which implies
\begin{thm}[Corollary 1.3 of \cite{Gupta14}]\Label{thm:dense}
For any $\tau \in \mathscr{T}$ and almost-every $L\in \mathcal{ML}$, the grafting ray $Gr_{tL}(\tau)$ ($t\geq 0$) descends to a dense ray  in moduli space.
\end{thm}
Here ``almost-every" refers to the Thurston measure on $\mathcal{ML}$, the space of measured laminations, so that $(\tau, L)$ corresponds to Teichm\"uler rays that are dense in $\M$ (\cite{Masur82}).
The density in Theorem \ref{thm:dense} yields the density in Theorem \ref{thm:thm1}.\\

In \cite{Gupta14}, the second author obtained the result in Theorem 1.1 for the case when $\rho$ is Fuchsian--- 
$\rho$ determines a hyperbolic structure $\tau$ on $S$, and a dense subset in $\mathscr{M}$ is given by ($2\pi$-)grafting $
\tau$ along multiloops and creating a sequence that approximates a dense grafting ray given by Theorem \ref{thm:dense}.

Let $C \cong (\tau, L)$ denote any projective structure on $S$ in Thurston coordinates. 
 We approximate the grafting ray by $2\pi$-grafting: 
\begin{theorem}\Label{t:EpDense}
For every $\ep > 0$, if  a measured lamination $M$ on $S$ is sufficiently close to $L$ in $\ML$, then for every sufficiently large $t > 0$, 
there is an admissible loop $N$ on $C$, such that, 
 $$d_{\T} (gr_N(C), gr_{t M}(\tau)) < \ep,$$
where $d_{\T}$ denotes the Teichm\"uller metric on $\mathscr{T}$.  
\end{theorem}

Let $\rho = \hol(C)$.
Then Theorem \ref{t:EpDense} and Theorem \ref{thm:dense} imply that $\P_\rho$ projects onto a subset in $\M$ whose $\ep$-neighborhood is $\M$:
Pick a generic lamination $M$ so that $(\tau, M)$ correspond to a dense Teichm\"uller ray in $\M$, so that Theorem \ref{thm:dense} holds. 
Since $\ep > 0$ is arbitrary, the projection of $\P_\rho$ must be dense in $\M$, which proves Theorem \ref{thm:thm1}.
Note that, in order to created a dense subset in $\M$,  we only used projective structures obtained by grafting $C$ along multiloops; in addition we can choose different $C$ in $\P_\rho$.

Our strategy is to generalize the argument for Fuchsian holonomy $\pi_1(S) \to \PSLr$ in \cite{Gupta14}, where the second author considers only grafting of hyperbolic structures, and uses the the Thurston metric on projective surfaces to construct  quasiconformal maps with small distortion. However, in this paper we consider grafting of a general projective structure $C$ with arbitrary holonomy $\rho\col \pi_1(S) \to \PSL$, and arguments here lie in three-dimensional hyperbolic geometry.
In particular, we introduce smoother piecewise Euclidean/Hyperbolic metrics on projective surfaces, which modify the Thurston metrics; 
then, in order to prove Theorem \ref{t:EpDense},  we construct piecewise smooth biLipschitz maps with the biLipschitz constant close to $1$.
This new metric is given by fat traintracks on the projective surfaces ``supported'' (\S \ref{s:Traintrack}) on round cylinders  on the Riemann sphere.
 Such fat traintracks were used in \cite{Baba_10-1}, where the first author relates different projective structures with common holonomy by $2\pi$-grafting. 
However our projective structures  $\Gr_N(C)$ and $\Gr_{t M}(\tau)$ in Theorem \ref{t:EpDense} typically do not share holonomy, and we are interested in their conformal structures.
Thus our treatment in this paper is different and more geometric.

\vspace{.1in}

We note that, as described in \cite{Goldman04}, the holonomy map $\hol\col \P \to \chi$ gives a ``resolution" of the mapping class group action on $\chi$.
 (The mapping class group action is $\hol$-equivariant and its action on $\P$ is discrete.)
 Thus it would be interesting if the holonomy fibers $\P_\rho$ tell us about the action on $\chi$.\\

\noindent \textbf{Outline of the proof of Theorem \ref{t:EpDense}.} 
If $M$ is sufficiently close to $L$, there is a {\it nearly straight} fat traintrack $T$ on $\tau$ carrying both $M$ and $L$ (\S \ref{defn:dim}). 
Then we can choose  a multiloop $N$ also carried by $T$ so that $L + N$ is  a good approximation of  $t M$,  compared on $T$ (Lemma \ref{lem:approx0}). 

Accordingly,  there are fat traintracks on the projective surfaces $\Gr_N(C)$ and $\Gr_{t M}(\tau)$ corresponding to $T$ on $\tau$ (\S \ref{S:CircularTraintrack}).
Then we construct certain Euclidean/\\hyperbolic metrics on the projective surfaces: 
 The fat traintracks enjoys Euclidean metrics, given by the carried geodesic measured laminations, and hyperbolic metrics in their complements, given by $\tau$. 
Then, with respect to those metrics, for sufficiently large $t$ , we construct a piecewise smooth bilipschitz map from $\Gr_N(C)$ and $\Gr_{t M}(\tau)$ with small distortion, which yields a  quasiconformal map between $\gr_N(C)$ and $\gr_{t M}(\tau)$ with  small distortion (\S \ref{s:proof}). \\

\textbf{Acknowledgements.}
The first author would like thank John Hubbard and Yair Minsky for bringing his attention to conformal structures of $\PP_\rho$. He acknowledges support from the European Research Council (ERC-Consolidator grant no. 614733).

The second author would like to thank Yair Minsky.
He would like to acknowledge support of the Danish National Research Foundation grant DNRF95 (Centre for Quantum Geometry of Moduli Spaces). 

Both authors acknowledge support from the GEAR Network (U.S. National Science Foundation grants DMS 1107452, 1107263, 1107367).

\section{Preliminaries}

\subsection{Teichm\"uller metric on Teichm\"uller space  and moduli space}
We refer to \cite{Ahlfors66, Hubbard06} for a background to Teichm\"{u}ller theory. In this paper, let  $\pi$ be the projection from the Teichm\"uller space $\mathscr{T}$ to the moduli space $\mathscr{M}$ obtained by forgetting markings.\\

Given two marked Riemann surfaces $X,Y$,  the  \textit{Teichm\"{u}ller distance} between them is 
$$ d_{\T}(X,Y) := \frac{1}{2} \log K $$
where $K$ is the minimum dilatation over all quasiconformal maps between $X$ and $Y$ that preserve the marking. This in fact defines a complete  Finsler metric on $\mathscr{T}$ invariant under the mapping class group. 
Thus it induces a metric topology on moduli space $\mathscr{M}$  that coincides with any other natural topology obtained from the other standard definitions.

\subsection{Laminations and traintracks}
(See \cite{CanaryEpsteinGreen84, Casson-Bleiler-88, Penner-Harer-92, Kapovich-01},  for background on measured laminations and traintracks.) 
A \textit{geodesic lamination} $\lambda$ on a hyperbolic surface is a closed subset that is foliated by simple geodesics ({\it leaves}).
A geodesic lamination is {\it minimal} if each leaf is dense in the lamination.  
On the other hand, a geodesic lamination is {\it maximal} if the complementary regions are ideal triangles.
 A basic example of a geodesic lamination is a \textit{multiloop} which is a union of disjoint simple closed geodesics.  
 Then,  given a geodesic lamination $\lam$, there is a unique decomposition into disjoint components:
\begin{equation*}
\lambda = \lambda_1 \sqcup \lambda _2 \sqcup \cdots \lambda_k \sqcup l_1\sqcup l_2 \sqcup \cdots \sqcup l_m \sqcup  \gamma
\end{equation*}
where $\lambda_i$ ($1\leq i\leq k$) are \textit{minimal} sublaminations, $\l_i$ ($1\leq i\leq m$) are isolated  bi-infinite geodesics that accumulate on the  other components, and $\gamma$ is a multiloop.\\
A \textit{measured lamination} is a geodesic lamination equipped with a measure on transverse arcs that is invariant under  homotopies of arcs preserving the transversality.

\begin{defn}
A \textit{traintrack (graph)} on a closed surface  is an embedded graph with $C^1$-smooth edges (\textit{branches}) such that
\begin{itemize}
\item
 each vertex has degree either two or three, 
\item at each vertex locally $C^1$-diffeomorphic to one of the figures in Figure \ref{f:TraintrackGraph}, so that is is a union of $C^1$-smooth segment, and
\item it complementary regions are not disks, mono-gons, or bi-gons. 
\end{itemize}

Vertices of a traintrack are called {\it switches}.

A \textit{weighted} traintrack is a traintrack with an assignment of non-negative real numbers to each branch, such that at any switch, they satisfy an \textit{switch} condition: $w_1 + w_2 = w_3$  or $w_4 =  w_5$ as in Figure \ref{f:TraintrackGraph}.
\end{defn}

\begin{figure}[H]
\begin{overpic}[scale=.3%, grid,tics=10
] {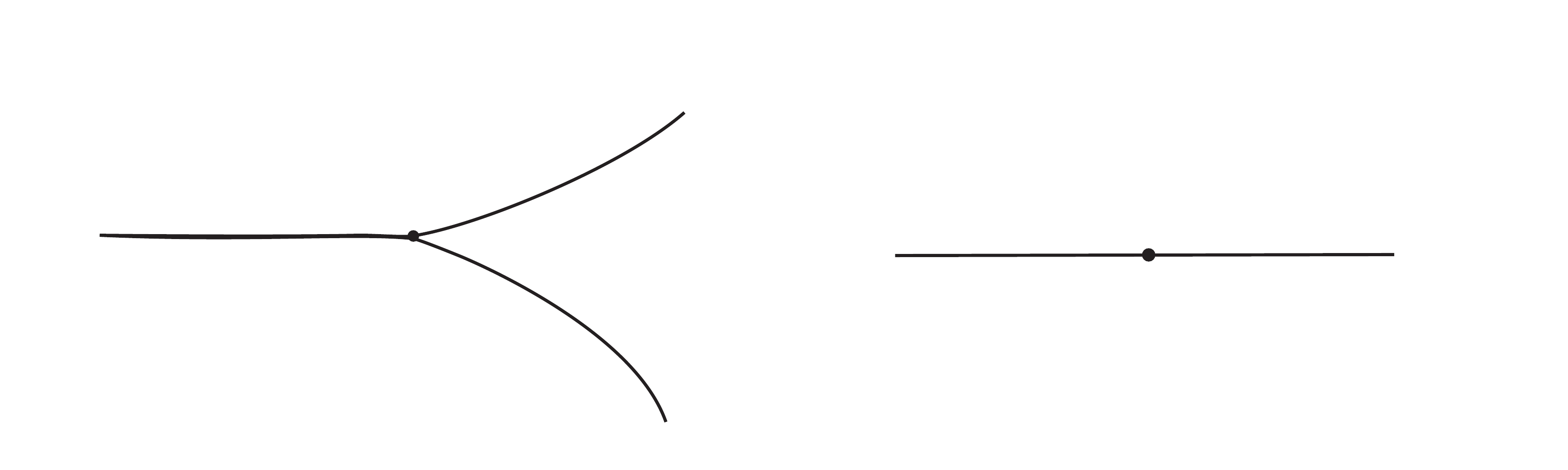} % figure file
\put(16,16){$w_3$}
\put(33,20){$w_1$}
\put(33,7){$w_2$}
\put(60,15){$w_4$}
\put(77, 15){$w_5$}
          %   \put( , ){}  
      \end{overpic}
\caption{}\label{f:TraintrackGraph}
\end{figure}

A weighted traintrack corresponds to a measured lamination on the surface.
Then we say that the measure lamination is {\it carried} by the traintrack.
 In the case when the weights are all integers, a weighted traintrack corresponds to a multiloop (of integer weight).\\

A traintrack is \textit{maximal} if its complementary regions are $C^1$-diffeomorphic to a triangle. 
Then  all possible weights satisfying the switch conditions on a fixed maximal traintrack is a convex cone in a Euclidean space of the dimension equal to the number of the branches. 
Then every measured lamination on the surface is carried by one of finitely many maximal traintracks. 
Thus the space of measured laminations $\ML$ on $S$ enjoys a structure of a piecewise linear manifold.

\subsubsection{Fat traintracks}

Given a rectangle $R$, in this paper, we always choose  a pair of its opposite edges to be horizontal and the other pair to be vertical.  
In addition we decide left and right vertical edges of it.  

A {\it fat traintrack} $T$ on a surface $S$ is a union of rectangles $R$ glued along its vertical edges so that a neighborhood of each vertical edge look like one of the pictures in Figure \ref{f:FatTraintrack}: 
 A vertical edge is either a union of two vertical edges or is identified to another vertical edge.
In the former case,  a {\it switch} point decomposes the long vertical edge into the two short edges. 
We often denote a fat traintrack as  $T = \{R_i\}$, where $R_i$ are rectangles.
Note that every fat traintrack converges to a traintrack graph by a $C^1$-smooth isotopy that collapses rectangles to edges so that the heights  of rectangles limit to zero. 
Thus, similarly, each rectangle $R$ of the traintrack is called a {\it branch}. 
If a vertical  edge of a branch is equal to or contained in a vertical edge of another branch, then we say that those two branches are {\it adjacent}.
Then the {\it boundary} $\bd T$ of  $T$ is the union of horizontal edges of rectangles, and it is a smooth except at switch points.   
Let $|T|$ denote the union of the rectangles of $T$.

\begin{figure}[H]
\begin{overpic}[scale=.3%, grid,tics=10
] {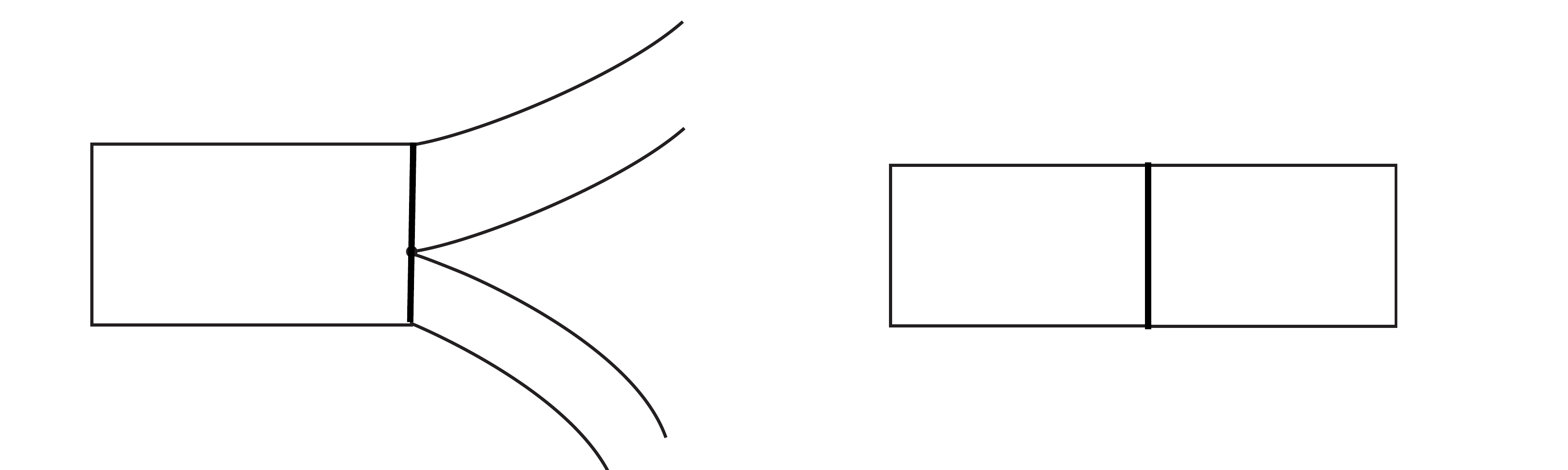} % figure file
          %   \put( , ){}  
      \end{overpic}
\caption{Vertical edges separating adjacent branches}\Label{f:FatTraintrack}
\end{figure}

A measured lamination $M$ on a surface is {\it carried} by a  fat traintrack $T$ if
\begin{itemize}
\item $M$ is contained in $|T|$, 
\item $M$ is transversal to vertical edges of rectangles, and
\item for each branch $R$ of $T$,  $M \cap R$ is  a union of arcs connecting opposite vertical edges of $R$.  
\end{itemize}
Then the weight of $M$ on a branch $R$ of $T$, is the transversal measure, given by $M$,  of a vertical edge of $R$. 

\subsection{Approximating measured laminations by multiloops}

Let $L$ be a measured lamination carried by a maximal traintrack $T$  with $n$ branches with (positive real) weights $(w_1,w_2,\ldots w_n)$, which satisfies the switch conditions.\\

By scaling a transversal measure by $t > 0$, we obtain a ray $t L = (t w_1, t w_2,\ldots t w_n)$ in $\ML$.
Then every point on this ray is approximated by a multiloop carried by the traintrack so that the differences of weights is uniformly bounded: 
\begin{lem}[Lemma 6.15 of \cite{Gupta14}]\Label{lem:approx0}
There exists a constant $D >0$, such that,  for every $t > 0$,  there is a multiloop $M_t$ carried by $T$ such that, letting  $(m_1(t),\ldots m_n(t)) \in \Z_{> 0}^n$ be the weights on $T$ representing $M_t$, 

\begin{equation}\Label{eq:differ}
- D < 2\pi m_i(t) -tw_i  < D
\end{equation}
for each $1\leq i\leq n$. 
\end{lem}

\subsection{Admissible loops and projective grafting}
Hyperbolic structures on surfaces are in particular projective structures. 
First we describe a grafting of a hyperbolic structure along a measured lamination.

Let $\exp\col \C \to \C \minus \{0 \}$ be the exponential map. 
Then, 
for $\theta > 0$,   a {\it crescent} of angle $\theta$ is a projective structure on $\R \times [0, \theta]$ given by restricting $\exp$. 
Then if $\theta < 2\pi$,  the crescent of angle $\theta$ is isomorphic to the region on $\RS$ bounded by two embedded circular arcs with common endpoints ({\it vertices}), intersecting at angle $\theta$.
If $\theta > 2 \pi$, the boundary of a crescent region similarly corresponds to such circular arcs but the crescent is immersed in $\rs$.  
The horizontal translations of $\R^2$ restrict projective isomorphisms of the crescent  $\R \times [0, \theta]$. 
Then given  $a >0$, there is a unique infinite cyclic group of such translations, so that the generator corresponds to a hyperbolic element in $\PSL$ of translation length $a$. 
Then the quotient of the crescent by the cyclic group is a projective structure on a cylinder, and  its boundary circles are isomorphic to a geodesic loop on a hyperbolic surface whose length is $a$. 

Let $\ell$ be a geodesic loop on a hyperbolic surface $\tau$. 
Then, for any $\theta > 0$,  then there is a corresponding projective cylinder obtained by quotienting a crescent of angle $\theta$ so that its boundary circles are isomorphic to $\ell$.
Thus we can insert the cylinder along the loop $\ell$, and we obtain new projective structure homeomorphic to $\tau$. 
This operation is a {\it grafting} of $\tau$ along $(\ell, \theta) \in \ML$.
If there are disjoint weighted closed geodesics, we can simultaneously intersect such cylinders of corresponding heights. 
A general measured lamination can be approximated by weighted multiloops.
Thus grafting of $\tau$ along a measured lamination is as a limit of grafting along multiloops. 
We refer to \cite{Kamishima-Tan-92, Epstein-Marden-87} for details.
\

Let $S$ be a closed orientable surface of genus at least two.
Then  any projective structure $C$ on $S$ is obtained by grafting a hyperbolic surface $\tau$ along a measured lamination $L$, and the choices of $\tau$ and $L$ are unique (\cite{Kamishima-Tan-92, Kullkani-Pinkall-94}).
Thus we denote as $C= Gr_L (\tau)$.
 We refer to $(\tau,L)$ as the \textit{Thurston coordinates} of the projective structure $C$, and write $C \cong (\tau, L)$.
Recalling the projection $\pi\col \P \to \T$,
 let  $gr_L(\tau) = \pi \circ Gr_L(\tau)$.

Note that above we graft along a measured lamination $L$ appearing circular on $\RS$.
In addition $\hol \Gr_{t L} (C) \in \chi$ changes when $t > 0$ changes .
Next we consider a variation of grafting along a multiloop on a (general) projective surface, which does not change holonomy.  

 Given a projective structure $C = (f, \rho)$, a  loop $\ell$ on $C$ is \textit{admissible} if  $\rho(\ell)$ is loxodromic and a lift  of $\ell$ to the universal cover embeds in $\RS$ by the developing map $f$. 
An \textit{admissible multiloop} is a union of a disjoint collection of such admissible loops.

If $\ell$ is admissible, the loxodromic $\rho(\ell)$ generates an infinite cyclic group in $\PSL$. 
Then, by a M\"obius transformation, we can  regarding $0$ and $\infty$ to be the fixed points of the cyclic group, so that $f$ embeds $\til{\ell}$ into $\C \minus \{0\}$.
By quotienting  $\C \minus \{0\}$ by the cyclic group, we obtain a projective structure $T_{\rho(\ell)}$ on a torus ({\it Hopf torus}).
Then $\ell$ is isomorphically embedded into $T_{\rho(\ell)}$. 
Similarly we can transform a projective structure $C$ to a new projective structure $\Gr_\ell(C)$ by inserting  the cylinder  $T_{\rho(\ell)} \minus \ell$ along $\ell$. 
In particular such grafting is often called {\it 2$\pi$-grafting}. 
It turns out that $\Gr_\ell(C)$ has also holonomy $\rho$ (see \cite{Goldman-87} and \cite{Kapovich-01}).
If there are disjoint admissible loops on $C$,  we can graft $C$ simultaneously along all loops and obtain a new projective structure, whose holonomy is indeed $\rho$. 
Similarly let $\Gr_M(C)$ denote the resulting projective structure, where $M$ is the union of the admissible loops. 

\subsubsection{Collapsing maps, pleated surfaces, and Thurston metric}
 See \cite{Kamishima-Tan-92, Kullkani-Pinkall-94}
(c.f. \cite{Baba_10-1}).  
Let $C \cong (\tau, L)$ be a projective structure on $S$ in Thurston coordinates. 
Then there is a corresponding measured lamination $\LL$ on $C$ consisting of circular leaves.
Moreover there is a natural {\it collapsing map} $\kap\col C \to \tau$ such that  $\kap$ preserves the marking of the surface and $\LL$ descends to $L$ by $\kap$. 
Given a  (measured) lamination on a surface, a {\it stratum} is a leaf of the lamination or a connected component of the complementary region. 
In particular $\kap$ takes each stratum of $(C, \LL)$ to a stratum of $(\tau, L)$.
Let $(\til{C}, \til{\LL})$ be the universal cover of  $(C, \LL)$. 
Then the developing map $f$ of $C$ embeds each stratum of $(\til{C}, \til{\LL})$ onto a region in $\RS$ bounded by circular arcs.

 The inverse-image  of $L$ on $\tau$ is a measured lamination $\til{L}$ on $\H^2$ invariant under $\pi_1(S)$. 
Then we obtain a {\it pleated surface} $\beta\col \H^2 \to \H^3$  bending $\H^2$ inside $\H^3$ along $L$ by the angles corresponding to its transversal measure of $\til{L}$.  
Then  $\beta$ is equivariant through $\hol(C) =: \rho$ (up to an element of $\PSL$), as $f$ is. 

Then, letting $\til{\kap}\col \til{C} \to \H^2$ be the lift of $\kap\col \C \to \tau$ to a collapsing map between the universal covers, 
 indeed $\til{\kap}$ relates the developing map $f\col \til{C} \to \RS$ and the pleated surface $\beta\col \H^2 \to \H^3$:
Namely, if $\kap$ diffeomorphically takes a stratum $\PP$ of $(\til{C}, \til{\LL})$ to a stratum $P$ of $(\H^2, \til{L})$, then, recalling that $f$ embeds $\PP$ into $\RS$,
$f(\PP)$ is diffeomorphic to $\beta(P)$ by the nearest point projection to a hyperbolic plane in $\H^3$ containing $\beta(P)$.

The {\it Thurston metric} on $C$ is a natural metric associated with the collapsing map $\kap\col (C, \LL) \to (\tau, L)$. 
It is the sum of a natural Euclidean/hyperbolic metric on $C$ and the transversal measure of non-periodic leaves of  $\LL$.
(See \cite{Kullkani-Pinkall-94}.)
In this paper we consider Thurston metric only on subsurfaces where the metric is hyperbolic. 

Let $M$ be the sublamination of $L$ consisting of the closed leaves $m_1, \dots, m_k$ of $L$, so that $M =  m_1 \sqcup \dots \sqcup m_k$.
Then  $\kap$ embeds $\kap^{-1}(M)$ onto $\tau \minus M$.
Then the Thurston metric on $C \minus \kap^{-1}(M)$ is hyperbolic so that $\kap$ is an isometric embedding. 
On the other hand $\kap^{-1}(m_i)$ is a cylinder foliated by parallel closed leaves of $\LL$ that are diffeomorphic to $m_i$ by $\kap$ for each $i \in 1, \dots, k$. 
Then $\kap^{-1}(m_i)$ enjoys a natural Euclidean metric so that its circumference is the length of $m_i$ and its height is the weight of $m_i$.
Thus $\kap^{-1}(m_i), \dots, \kap^{-1}(m_i)$  are disjoint Euclidean cylinders in $C$. 
Therefore the restriction of $\kap$ to each stratum $P$ of $(C, L)$ is an isometric embedding.

\section{Fat traintracks and laminations}\Label{TraintrackLamination}

Let $X, Y$ be metric spaces with distances $d_X, d_Y$ and  $\phi\col X \to Y$  be a continuous map.
Then, $\phi$ is {\it $A$-bilipschitz} for  $A \geq 1$, if $$\frac{1}{A}  d_X(p, q) < d_Y(\phi(p), \phi(q)) < A d_X(p,  q),$$
for all $a, b \in X$.
Given $B \geq 0$,  $\phi$ is a  {\it $B$-rough isometry}  if $$d_X(p, q) - B < d_Y(\phi(p), \phi(q)) <  d_X(p, q) + B.$$
For $A \geq 1, B > 0$, $\phi$ is an {\it $(A, B)$-bilipschitz} map $\phi$  is  an $A$-bilipschitz map and also a $B$-rough isometry. 
In addition, throughout this paper, $(A, B)$-bilipschitz maps are  piecewise differentiable.

\begin{defn}[Nearly straight traintracks]\Label{defn:dim}
For $\epsilon>0$ and $w > 0$, a fat traintrack on a hyperbolic surface is \textit{$(\epsilon, w)$-nearly straight},  if each branch is smoothly $(1+\epsilon)$-rough isometric to some euclidean rectangle of width at least $w$ and at each switch point, the angle between horizontal edges meeting at a switch point is less than $\epsilon$. 
\end{defn}
Here the {\it width} of a rectangle is the distance between the vertical edges. 
Note that, given $w > 0$, if $\ep > 0$ is sufficiently small, then each $(\ep, w)$-nearly straight traintrack almost looks like a geodesic lamination.

\begin{defn}[Angles between laminations;  \cite{Baba_10-1}]\label{defn:angle}
Let $\ell$ and $n$ are simple geodesics on a hyperbolic surface $\tau$.
If $\ell$ and $n$ intersect at a point $p$,  
let $\angle_p(\ell, n)$ be the angle, taking a value in $[0, \pi/2]$, between $\ell$ and $n$ at $p$.
More generally let $\lambda$ and $\nu$ be geodesic laminations on  $\tau$. 
The angle $\angle_\tau(\lambda, \nu)$ be the supremum 
of $\angle_p(\ell_p, n_p)$ over all points $p$ of the intersection $|\lambda| \cap |\nu|$, where $\ell_p$ and $n_p$ are the leaves of $\lambda$ and $\nu$ intersecting at $p$. 
\end{defn}
To determine angles defined in  Definition \ref{defn:angle}, we always take geodesic representatives of loops and laminations even if they are given as  topological loops or laminations. 

Then it follows from basic hyperbolic geometry: 
\begin{lem}\label{lem:angle}
Let $T$ be a  fat traintrack  on a hyperbolic surface.
Let  $\mu, \nu$ be laminations carried by $T$.
For any $\epsilon>0$, there is  $\delta>0$ such that if $T$ is $(\del, \ep)$-nearly straight then $\angle_\tau(\mu, \nu) < \epsilon$.
\end{lem}

Let $\tau$ be a closed hyperbolic surface. 
Consider an essential loop on $\tau$ of shortest length. 
Let $K$ be the one-third of its length. 
Let $\lam$ be a geodesic lamination on $\tau$. 
Then
\begin{lemma}[c.f. \cite{Baba_10-1}]\Label{StraightTraintrack}
 For every $\ep > 0$, there is an $(\ep, K)$-nearly straight fat traintrack $T$ carrying $\lam$ on $\tau$.
\end{lemma}

Then we have
\begin{lemma}\Label{l:TrainintrackNbhd}
Let $\tau$ be a hyperbolic surface.
Consider the shortest geodesic loop on $\tau$ and  let $K > 0$ be the one third of its length.
Then for every measured lamination $L$ on $\tau$ and $\ep > 0$, there is a neighborhood $U$ of $L$ in $\ML$, such that for every $M$ in $U$, there is an $(\ep, K)$-nearly straight traintrack on $\tau$ carrying $L$ and $M$. 
\end{lemma}

\begin{proof}
Suppose, to the contrary, that there is no such a neighborhood $U$. 
There is a sequence of measured laminations $M_i$ that converges to $L$ such that,  for each $i$, there is no $(\ep, K)$-nearly straight traintrack on $\tau$ caring both $L$ and $M_i$. 
Then, since the space $\GL$ of geodesic laminations on $S$ is compact, we can in addition assume that $|M_i|$ converges to a geodesic lamination $\lam$ by taking a subsequence. 
Then since $M_i \to L$, thus $\angle_\tau(L, \lam) = 0$.
Therefore $\lam \cup |L|$ is a geodesic lamination on $\tau$. 
Thus, by Lemma \ref{StraightTraintrack}, there is an $(\ep, K)$-nearly straight traintrack $T$ carrying $\lam \cup |L|$.
Clearly $T$ carries $L$. 
In addition, since $M_i$ converges to $L$, thus $M_i$ is carried by $T$ for large enough $i$\,---\,this is a contradiction.
\end{proof}

%%%

\subsection{Fat traintracks on projective surfaces}\Label{s:Traintrack}

Given a projective structure $C \cong (\tau, L)$ in Thurston coordinates, there is a nearly straight traintrack $T$ on $\tau$ carrying $L$ by Lemma \ref{StraightTraintrack}.
In this section we construct a corresponding ``nearly circular" traintrack on $C$, as used in \cite{Baba_10-1, Baba_10-2}.

\begin{defn}A \textit{round cylinder} is  a cylinder $\AA$ embedded in the Riemann sphere $\RS$ bounded by round circles.  
Recalling that $\bd \H^3 = \RS$,  there is, in $\H^3$,  a unique geodesic orthogonal to the hyperbolic planes bounded by boundary circles of $\AA$--- we call this geodesic the {\it axis} of $\AA$.
Then this axis contains a geodesic segment that connects the hyperbolic planes, which we call the {\it core} of $\AA$. 
Then $\AA$ is foliated by round circles  bounding disjoint copies hyperbolic planes orthogonal to the axis ({\it vertical foliations of $\AA$}). 
\end{defn}
\textit{Remark.} Let $m$ be the modulus of $\AA$. Then  $2\pi m$ is the distance between the hyperbolic planes bounded by boundary circles of $\AA$.  

\begin{defn}[Canonical metrics on  round cylinders] \label{def:EuclideanCylinder}
Let $\AA$ be a round cylinder on $\RS$ and  let $g$ be the axis of $\AA$.
Then $\RS$ minus the end points of $g$ is biholomorphic to an infinite Euclidean cylinder.
 Its universal cover is biholomorphic to $\C$, which is biholomorphic to the Euclidean plane. 
Thus this Euclidean metric induces a Euclidean metric on $\AA$.
Since its Euclidean metric is unique up to scaling, we normalize it so that each round circle of the vertical foliation of $\AA$ has length $2\pi$.
\end{defn}

An arc $\al$ is {\it supported} on a round cylinder  $\AA$  on $\RS$ if it is properly embedded in $\AA$  and transversal to the vertical foliation of $\AA$, so that it connects the boundary components of $\AA$.

Let $R$ be a projective structure on  rectangle.
Then $R$ is supported on a round cylinder $\AA$ on $\RS$ if $dev(R)$ is into $\AA$, its horizontal edges are supported on $\AA$ and its vertical edges immerses into boundary circles of $\AA$ (Figure \ref{f:Supported}).
In particular $R$ is a {\it circular rectangle} supported on $\AA$ if its horizontal edges are orthogonal to the vertical foliations of $\AA$.

\begin{figure}[H]
\begin{overpic}[scale=.25, %grid,tics=10
] {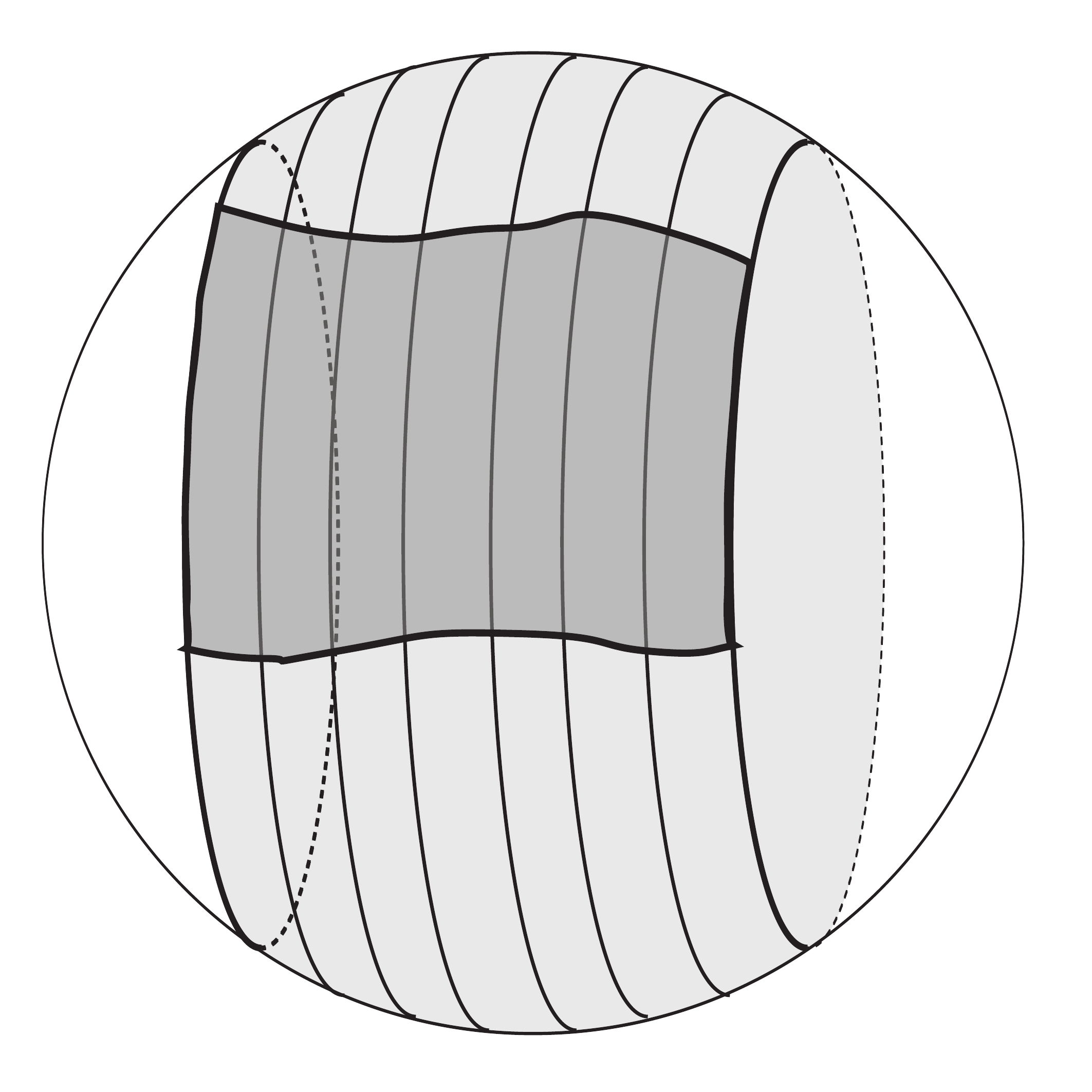} % figure file
         \put(8, 50){$\RS$}
         \put(47, 84){$\AA$}
         \put(40, 60){$\RR$}
          %   \put( , ){}  
      \end{overpic}
\caption{A rectangle supported on a round cylinder (embedded case)}\label{f:Supported}
\end{figure}

\begin{defn}[Foliation and metric on rectangles]\Label{DefEuclideanRectangle} 
Let $R$ be a projective structure on a rectangle supported on a round cylinder $\AA$. Let $\FF$ be the canonical foliation of $\AA$ by round circles. 
Let $e_1$, $e_3$ be its horizontal edges and  $e_2, e_4$ its vertical edges, so that, via ${\rm dev}(R)$, $e_1$ and $e_2$ immerse into different boundary circles of $\AA$ and $e_2$ and $e_4$ are transversal to $\FF$.  
Then we can pull back $\FF$ to a foliation of $R$ (\textit{vertical foliation}).  
Similarly $R$ enjoys a natural Euclidean metric by pulling back the Euclidean metric on $\AA$ via  $dev(R)$.
\end{defn}

Note that if a rectangle $R$ is supported on a round cylinder, then with the canonical Euclidean metric, $R$ isometrically embeds in $\R^2$ (so that the leaves of  its vertical foliation maps to vertical lines).

\begin{definition}[Nearly circular arcs]
Let $\al$ be an arc supported on a round cylinder $\AA$, and let $\ep > 0$,  
Then $\al$ is $\ep$-nearly circular if $\al$ is contained in a circular rectangle of height at most $\ep$ and it intersects each vertical leaf of $\AA$ at an angle $\ep$-close to $\pi/2$.
\end{definition}

A fat traintrack $\TT = \{\RR_i\}$ on a projective surface is {\it supported} on round cylinders if each branch $\RR_i$ is supported on a round cylinder on $\RS$ by the developing map. Then

\begin{theorem}[\cite{Baba_10-1} Lemma 7.2.]\Label{AdmissibleTraintrack}
Let $\TT$ be a fat traintrack on a projective surface.
If $\TT$ is supported on round cylinders, then every loop carried by $\TT$ is admissible. 
\end{theorem}

\begin{definition}[Nearly circular rectangles]\label{defn:almostCirc}
Let $R$ be a rectangular projective surface supported on a round cylinder $\AA$. 
The rectangle $R$  is \textit{circular} if it horizontal edges are circular arcs on $\AA$ that are orthogonal to $\FF$. 

More generally, 
for $\ep > 0$ and $K > 0$, the rectangle $R$ is $(\ep, K)$-{\it nearly circular}, if 
\begin{itemize}
\item the core of $\AA$ has length at least $K$, and
\item each horizontal edge of $R$ is $\ep$-circular.\end{itemize}
\end{definition}

Note that, by the second condition, if $R$ is an $(\ep, K)$-nearly circular rectangle, then the Euclidean lengths of vertical leaves of $R$ are $2\ep$-close.

\begin{definition}[Nearly concentric cylinders]
Let $\AA_1$ and $\AA_2$ are round cylinders on $\RS$.
Let $c_1$ and $c_2$ be their cores.  
Suppose that $\AA_1$ and $\AA_2$ are {\it adjacent} along a loop $\ell$, i.e. $\AA_1$ and $\AA_2$ have disjoint interiors and share a boundary loop $\ell$.
Then the hyperbolic plane bounded by $\ell$ contains an endpoint of $c_1$ and an endpoint of $c_2$.
Then $\AA_1$ and $\AA_2$ are {\it concentric} if  $c_1 = c_2$.
More generally, for $\ep > 0$,  $\AA_1$ and $\AA_2$ are {\it $\ep$-nearly concentric}, if the endpoints are $\ep$-close in the hyperbolic plane. 
\end{definition}

\begin{definition}[Nearly-circular traintrack]
A fat traintrack $\TT$ on a projective surface  is $(\ep, K)$-nearly circular if each branch of $\TT$ is $(\ep, K)$-nearly circular and round cylinders supporting adjacent branches of $\TT$ are $\ep$-nearly concentric.
\end{definition}

\subsubsection{Existence of nearly circular traintracks}\Label{S:CircularTraintrack} 

Let $C \cong (\tau, L)$ be a projective structure on $S$.
Let $\lam$ be a geodesic lamination containing $|L|$.
Let $K > 0$ be a constant that does not exceed the one-third of the length of shortest leaf of $\lam$. (In the case that $\lam$ contains no closed leaf,  fix arbitrary $K > 0$.)
Then, for every $\ep > 0$ there is an $(\ep, K)$-nearly straight fat traintrack $T_{\ep, K}$ on $\tau$ carrying $\lam$ (Lemma \ref{StraightTraintrack}).
In the following discussion,  by `` sufficiently straight  $T$"  we mean an $(\ep, K)$-nearly straight traintrack $T$ with sufficiently small $\ep > 0$ while $K$ is fixed by $C$. 

Let $\kap \colon C \to \tau$ be the collapsing map. 
Let $\LL$ be the measured lamination on $C$ that descends to $L$ via $\kap$.
We say that a fat traintrack on $\TT$ on $C$ {\it descends} to $T$, if  $\kap(|\TT|) = |T|$ as subsets of $\tau$, and  $\kap$ bijectively takes branches  $\TT$ to branches $T$ up to a small perturbation of vertical edges. 
Then 
\begin{thm}\Label{ThmTraintrack}
For every $\ep > 0$ and $K > 0$, there is $\del > 0$ such that  
given a projective structure $C \cong (\tau, L)$ on $S$ and  a $(\del, K)$-nearly straight  fat traintrack $T = \{ R_i \}$ on $\tau$ carrying $L$, 
then there is a fat traintrack $\TT = \{ \RR_i\}$ on $C$  that descends to $T$, up to  $\ep$-small isotopy of vertical edges, such that  
\begin{enumerate}
\item  if $R$ and $\RR$ are corresponding branches of $T$ of $\TT$, respectively, 
then
\begin{enumerate}
\item $\RR$ is supported on a round cylinder on $\RS$ and the core of the round cylinder has length $\ep$-close to the width of the branch $R$, and
\item each vertical leaf of $\RR$ has length $\ep$-close to the weight of $L$ on $R$.  
\end{enumerate}
\item    $\TT$ is $(\ep, K)$-nearly circular.\Label{i:height}
\end{enumerate}
\end{thm}

In \cite{Baba_10-1}, the first author obtained a similar statement simultaneously for a certain family of projective structures with fixed holonomy.
Theorem \ref{ThmTraintrack} essentially follows from arguments in \cite[\S 7.2]{Baba_10-1}. 
We here outline how to apply the arguments to our setting.

{\it Outline of the proof.} 
Let $\kap\col C \to \tau$ be the collapsing map.
Let $\til{\kap} \col \til{C} \to \H^2$ be its lift to the collapsing map between their universal covers. 
Let $\til{T}$ and $\til{\TT}$ be the lifts of $T$ to $\h^2$ and $\TT$ to $\til{C}$, respectively. 
Let $\beta \col \H^2 \to \H^3$ be the pleated surface induced by $(\tau, L)$.
For  each switch point $p$ of $\til{T}$, since $p$ is disjoint from $\til{L}$, there is a unique point  ${\tt p} \in \til{C}$ mapping to $p$ by $\til{\kap}$. 
Then, 
for every $\ep > 0$, if $T$ is sufficiently straight, then,  we can pick a round circle $c_p$ on $\RS$ for each switch point $p$ of $\til{T}$ such that 
\begin{enumerate}
\item ${\tt p}$ is contained in $c_p$, 
\item the hyperbolic plane $H_p$ in $\H^3$ bounded by $c_p$ contains $\beta(p)$,
\item if a stratum $P$ of $(\H^2, \til{L})$ intersects a vertical edge of $\til{T}$ containing $p$, then $\beta(P)$ intersects $H_p$ transversally at an angle $\ep$-close to $\pi/2$ \Label{i:NearlyOrthgonal}, and
\item the choice of  $c_p$'s is $\rho$-equivariant.  \Label{i:Equivariant}
\end{enumerate}
Such a family of round circles $\{c_p\}$ can be constructed, since $T$ is nearly straight, using the relation between $dev(C)$ and $\beta$. 
Recall,  if $\PP$ is the stratum of $(\til{C}, \til{\LL})$ corresponding to $P$, that  $dev(C)$ embeds $\PP$ into $\RS$.
Then, by  (\ref{i:NearlyOrthgonal}), if $\PP$ is a circular arc, then $\PP$ intersects $c_p$ in a single point at an angle $\ep$-close to $\pi/2$.
Similarly, if $\PP$ has interior, then $\PP$ intersects $c_p$ in a circular arc  {\it nearly orthogonally}, that is,  exactly two circular arcs of $\bd \PP$ intersect $c_p$  at angles $\ep$-close to $\pi/2$.
Since each branch of $T$ has width at least $K$, if $\ep > 0$ is sufficiently small, then for each branch $R$ of $\til{T}$, the round circles corresponding to the switch points in its  vertical edges are disjoint by  (\ref{i:NearlyOrthgonal}).
Let $\AA_R$ be the round cylinder bounded by the round circles.
Then $\beta(R)$ is $\ep$-close to the core of  $\AA_R$ by (\ref{i:NearlyOrthgonal}) since $T$ is sufficiently straight.

 Then by (\ref{i:Equivariant}),  this correspondence between branches $R$ of $\til{T}$ and round cylinders $\AA_R$ is $\rho$-equivariant. 
Moreover, if $T$ is sufficiently straight, 
\begin{enumerate}
\item if $R_1$ and $R_2$ are adjacent  branches of $\til{T}$, then  corresponding round cylinders $\AA_{R_1}$ and $\AA_{R_2}$ are adjacent, 
\item the core of $\AA_R$ has length at least $K/2$,
\item if $P$ is  a stratum of  $(\H^2,  \til{L})$ intersecting a branch $R$ of $\til{T}$, then letting $\PP$ be the corresponding stratum of $(\til{C}, \til{L})$, $\PP$ is $\ep$-nearly orthogonal to the boundary circles of $\AA_R$, and $\beta(P)$ is $\ep$-nearly orthogonal to hyperbolic planes bounded by $\AA_R$. \Label{it:CoreBranch}
\end{enumerate}

Next we shall decompose a subset  $\til{\kap}^{-1}(|\til{T}|)$ of $\til{C}$ into rectangles, producing a traintrack supported on the round cylinders $\RR_R$. 
Suppose that a stratum $P$ of  $(\H^2, \til{L})$  is disjoint from horizontal edges of $R$.
Then,  if $P$ is a geodesic, $P \cap R$ is a geodesic segment, and otherwise $P$ has positive area and $P \cap R$ is a rectangle  bounded by vertical edges of $R$ and two leaves of $\til{L}$.
Let $f = dev(C)$ and let $\PP = f^{-1} (\AA_R) \cap \til{\kap}^{-1}(P)$.
Then, if $T$ is sufficiently straight, then $\PP$ is an $(\ep, K)$-nearly circular rectangle supported on $\AA_R$.
If the boundary of $P$ contains a horizontal edge of $R$, then  $P \cap R$ is a rectangle bounded by vertical edges of $R$, a leaf of $L$ and a smooth segment in $\bd \til{T}$.
Then let $\PP$ be the corresponding component of $\kap^{-1}(P) \cap  f^{-1}(\AA_R)$ that is a rectangle supported on $\AA_R$. 
Let $\RR_R$ be the union of all such $\PP$ over all strata $P$ of $(\H^2, \til{L})$ intersecting $R$.
Then  $\RR_R$ is a rectangle supported on $\AA_\RR$. 
Then (\ref{it:CoreBranch}) implies that $\til{\kap}$ takes $\RR_R$ to $R$ up to an $\ep$-small perturbation of vertical edges.

Then, by applying this construction to all branches $R$ of $\til{T}$, we see that  $\til{\kap}^{-1}(|\til{T}|)$ decomposes into such rectangles $\RR_R$ supported on corresponding round cylinders $\AA_R$.
 Then we obtain  this structure of a traintrack on   $\til{\kap}^{-1}(|\til{T}|)$  that is diffeomorphic to $\til{T}$ and it carries $\til{\LL}$.
  Since it is invariant under $\pi_1(S)$,   we obtained a traintrack $\TT$ on $C$ carrying $\til{\LL}$.

From this construction, in  Theorem \ref{ThmTraintrack}, we have in addition   

\begin{proposition}\Label{CloseToCore}
\begin{enumerate}
\item  If $\ell$ is a leaf of $\til{\LL}$ passing a branch $\RR$ of $\til{\TT}$, then $\ell \cap \RR$ is an $\ep$-circular arc supported on  $\AA_\RR$, and 
$ \beta\ci \til{\kap}$ takes  the arc $\ell \cap \RR$ to a curve $\ep$-close to the core of $\AA_\RR$. 
\item  Similarly if $\ell$ is a smooth segment of $\bdr \til{\TT}$ intersecting $\RR$, then,  $\ell \cap \AA_\RR$ is an $\ep$-nearly circular arc supported on $\AA_\RR$, and
$\beta \ci \til{\kap}$ takes  the arc $\ell \cap \RR$ to a curve $\ep$-close to the core of $\AA_\RR$. 
\end{enumerate}
\end{proposition}
\Qed{ThmTraintrack}

\section{The proof of Theorem \ref{t:EpDense}}\Label{s:proof}
Let $C \cong (\tau, L)$ be a projective structure on $S$ in Thurston coordinates.\\
We prove

\vspace{2mm}
\noindent{\bf Theorem \ref{t:EpDense}.}
For every $\ep > 0$, if   $M$  is sufficiently close to $L$ in $\ML$, then for every sufficiently large $t > 0$, 
there is an admissible loop $N$ on $C$, such that, 
 $$d (gr_N(C), gr_{t M}(\tau)) < \ep.$$ 
\vspace{2mm}

Consider the shortest closed geodesic loop on $\tau$, and let $K$ be the one-third of its length. 
Then by  Lemma \ref{l:TrainintrackNbhd}, for every $\ep > 0$,  if $M$ is a measured geodesic lamination on $S$ and it is sufficiently close to $L$, 
 then there is an $(\ep, K)$-nearly straight traintrack $T$ on $\tau$ carrying both $L$ and $M$. 
 As before, by ``sufficiently straight $T$'', we mean that $\ep > 0$ is sufficiently small, while $K$ is fixed, which is realized by taking $M$ sufficiently close to $L$. 
 \\

Let $\LL$ be the circular measured lamination on $C$ that descends to $L$ by the collapsing map $\kap\col C \to \tau$.
If $M$ is sufficiently close to $L$, then by Theorem \ref{ThmTraintrack},  there is in addition an $(\ep, K)$-nearly circular traintrack $\TT$ on $C$ carrying $\LL$ such that  $\TT$ is admissible and descends to $T$, up to $\ep$-small homotopies of vertical edges.\\

For $t \geq 0$, let 
$$C_t = Gr_{tM}(\tau).$$
Since $T$ is $(\ep, K)$-nearly straight traintrack carrying  $M$, if $M$ is sufficiently close to $L$, then
let $\TTT_t$ be the $(\ep, K)$-nearly circular traintrack on $C_t$ that descends to $T$ (Theorem \ref{ThmTraintrack}). 
Let $\{ \RRR_i\}$ denote its branches.

Note that  $\TT$ and $\TTT_t$ are diffeomorphic to $T$ as $C^1$-smooth traintracks via (perturbations of) collapsing maps. 
By Lemma \ref{lem:approx0}, there is a constant $D > 0$, such that for every $t > 0$ there is a multiloop $N\, (= N_t)$ on $C$ carried by $\TT$ such that the weights of $N + L$ and $t M$ are $D$-close on each branch. \\

Since $N$ is carried by the fat traintrack $\TT$ supported by round cylinders, $N$ is admissible, by Theorem \ref{AdmissibleTraintrack}. Let $$C_N =  \Gr_N(C).$$\\
Since $N$ is carried by $\TT$,  thus  $$ \Gr_N(|\TT|)$$ is a subsurface of $C_N$.
Each  branch  $\RR$  of $\TT$ is an $(\ep, K)$-nearly circular rectangle supported on a round cylinder.
Thus $N \cap \RR$ is an admissible multiarc on $\RR$ (see \cite{Baba_10-1}) and the graft of $\RR$ along $N \cap \RR$ is an $(\ep, K)$-nearly circular rectangle supported on the same cylinder. 
With the induced Euclidean metric, this grafts increases the length of each vertical leaf by $2 \pi$ times the number of the arcs of $N \cap R$.
Then  $$ \Gr_N(|\TT|)$$  decomposes into the grafts of each branches $\RR$ of $\TT$, and yields a fat traintrack denoted by  $\TT_N$.   
Therefore  $\TT_N$ is an $(\ep, K)$-nearly circular traintrack diffeomorphic to $\TT$ such that their corresponding branches are supported on the same round cylinders. 

Then there is $\hat{\kap}_N\col C_N \to C$ be the natural collapsing map, such that it collapses grafting cylinders onto $N$ and thus it realizes  the inclusion $C \minus N \sub C_N$. 
Let $\kap_N\col C_N \to \tau$ be the composition of the collapsing maps $\kap_C$ and $\hat{\kap}_N$. 
By Theorem \ref{ThmTraintrack} and the relation between $\TT$ and $\TT_N$, discussed above, we have 
\begin{prop}\Label{p:TN}
For every $\ep > 0$, if $T$ is sufficiently straight, then, for every $t > 0$,
$\TT_N$ is an $(\ep, K)$-nearly circular fat traintrack on $C_N$, such that 
\begin{enumerate}

\item $\kap_N \col C_N \to \tau$ descends $\TT_N$ to $T$ up to a ($\ep$-)small homotopies of vertical edges, 

\item if $R$ is  branch of $T$ and $\RR$ is the corresponding branch of $\TT_N$ supported on the round cylinder $\AA$, then

\begin{enumerate}
\item the length of the core of $\AA$ is $\ep$-close to the width of $R$.
\item  each  vertical leaf of $\RR$ has Euclidean length  $\ep$-close to the weight of $L +N$ on the branch $R$.
\end{enumerate}

\end{enumerate}

\end{prop}

We will show the inequality in Theorem \ref{t:EpDense}  by indeed constructing a piecewise-smooth $(1 + \ep)$-quasiconformal map from $C_N$ to $C_t$ that preserves the marking of the surface.
Here is an outline of the construction (for any given $\ep > 0$).   
\begin{enumerate}
\item[\textit{Step 1}.]
 The complement of the traintracks $\TT_N$ and $\TTT_t$ in $C_N$ and $C_t$ are hyperbolic with Thurston metric on $C_N$ and $C_t$, and, there is a natural hyperbolic isometry $$\psi\col  C_t \minus  \TTT_t \to C_N \minus \TT_N.$$ \Label{item:complement}
 \item[\textit{Step 2}.] \Label{item:TraintrackMap}
 We next construct a quasiconformal map  $$\phi\col \TT_N \to \TTT_t$$ that has small distortion if $t > 0$ is sufficiently large. 
 This map is given by natural ``linear" maps between corresponding branches with respect to appropriate Euclidean metrics obtained by slightly modifying the canonical Euclidean metrics. 
 In particular,  there is $H > 0$, 
 if $T$ is sufficiently straight, such that this map is $(1 + \ep, H)$-bilipschitz  for sufficiently large $t > 0$.  \\

\item[\textit{Step 3}.]  \Label{item:adjust}
We modify $\phi$ so that it coincides with $\psi$ along $\bd \TT_N$ so that $\phi$ and $\psi$ define a (continuous) $(1 + \ep, H)$-bilipschitz map  $C_N \to C_t$.

Since $\phi$ is in particular homeomorphism $\TT_N \to \TTT_t$, it restricts to a homeomorphism  $\bdr \phi\col \bdr \TT_N\to \bdr \TTT_t$.
The complements  $C_N \minus \TT_N$ and $C_t \minus \TT_t$ are identified by the hyperbolic isometry $\psi$, which restricts to an isometry from $\bd \TT_N$ to $\bd \TT_t$.
 Then $\bdr \phi$ is an $(1+\ep, \ep)$-bilipschitz map  with the complementary hyperbolic metrics. 
Lastly adjust $\phi\col \TT_N \to \TTT_t$ near $\bdr |\TT_N|$ by postcomposing with an $(1+ \ep, \ep)$-bilipschitz map, so that $\phi$ and $\psi$ match up on the boundary $\bd \TT_N$. \\
\end{enumerate}

Therefore, if $T$ is sufficiently straight and  $t > 0$ is sufficiently large, then we have a $(1+ \ep, K)$-bilipschitz map  $C_N \to C_t$ given by  $\psi$ and  (modified)  $\phi$ with respect to appropriate Euclidean/Hyperbolic metrics.

\subsection{Step 1: Isometry between the complements of traintracks}

Recall that the collapsing maps $\kap \col C \to \tau$ and $\kap_t\col C_t \to \tau$ take $|\TT|$ and $|\TTT_t|$ to the nearly straight traintrack $|T|$ carrying both measured laminations $M$ and $L$.
Thus Thurston metrics on $C_t \minus \TT_t$ and  $C \minus \TT$ are hyperbolic, and moreover
$C_t \minus \TT_t$ is  isometric to $C \minus \TT$ via $\kap^{-1} \ci \kap_t$.

Recall the collapsing map $\hat{\kap}_N\col C_N \to C$.
Since $N$ is carried by $\TT$, thus $\hat{\kap}_N$ restricts to an isomorphism from $C_N \minus \TT_N$ to $C \minus \TT$ as projective surfaces.
Then
$C_N \minus \TT_N$ enjoys a hyperbolic metric obtained by pulling back the hyperbolic metric on $\tau$ by the composition of collapsing maps, $\kap_N = \kap \circ \hat{\kap}_N \col C_N \to \tau$.
Then
let $\psi\col C_N \minus \TT_N \to  C_t \minus \TTT_t$ be the natural hyperbolic isometry give by $\kap_t^{-1} \ci \kap_N$.

\subsection{Step 2: Linear map between traintracks}\Label{s:map}

\subsubsection{Linear maps between rectangles supported on round cylinders}\Label{defRectangleMap}

Let $\AA, \AA' $ be round cylinders on $\RS$. 
Let $R$ and $R'$ are projective structures on a rectangle supported on $\AA$ and $\AA'$, respectively.
Then there are canonical Euclidean metrics $E$ on $R$ and $E'$ on $R'$ by Definition \ref{DefEuclideanRectangle}. 
Then  there is a unique diffeomorphism $\xi\col R \to R'$ such that
\begin{itemize}
\item $\xi$ preserves the left vertical edge, right vertical edge, and horizontal edges, 
\item $\xi$ takes each vertical leaf of $R$ onto a vertical leaf of $R'$ linearly; thus $\xi$ induces a continuous map $\xi_\ast$  in the horizontal coordinate, and 
\item $\xi_\ast$  is also linear. 
\end{itemize}

Then
\begin{lemma}\Label{lem:CompareBranches}
For every $\ep > 0$ and $K >0$, there is $\del > 0$, such that, if

\begin{enumerate}
\item   $R$ and $R'$ are $(\del, K)$-nearly circular, \Label{it:NearlyCircular}
\item $\del$ bounds, from above, the difference of the lengths of the cores of $\AA$ and $\AA'$, and
\item  there are constants $V, V' > 1/\del$ with $-K/2 < V - V' < K/2$, such that, for  all vertical leaves of $\ell$ and $\ell'$ of $R$ and $R'$, respectively, \Label{i:height} 
\end{enumerate}

$$ -\del <  \length_{E}(\ell) - V < \del  $$
$$-\del  < \,\length_{E'}(\ell') - V'\, < \del,$$

\Label{it:AlmostConstant}

then $\xi\col R \to R'$ is a smooth  $(1 + \ep, K)$-bilipschitz map.
\end{lemma}

\Proof
First we consider $\xi$ in  the vertical direction. 
 If $\del > 0$  is sufficiently small, by (\ref{i:height}), for all  vertical leaves $\ell$ of $R$ and $\ell'$ of $R'$ corresponding by $\xi$,
  $\length_E (\ell)$ and $\length_{E'}(\ell')$ are  sufficiently large and $- K < \length(\ell) - \length(\ell') < K$.
Thus  $\xi| \ell$ is  a smooth $(1 + \ep, K)$-bilipschitz map for all vertical leaf $\ell$ of $R$.

Next we consider in the horizontal direction. 
Let $W$ and $W'$ denote the Euclidean widths of $R$ and $R'$, respectively. 
Then $\xi_\ast\col [0,W] \to [0, W']$ denotes the linear map induced by $\xi$. 
Then, if $\del > 0$ is sufficiently small, since $|W - W'| < \del$ and $W, W' > K$, then $\xi_\ast$ is a $(1 + \ep, \ep)$-bilipschitz map.

\begin{figure}[H]
\begin{overpic}[scale=.2,% grid,tics=10
] {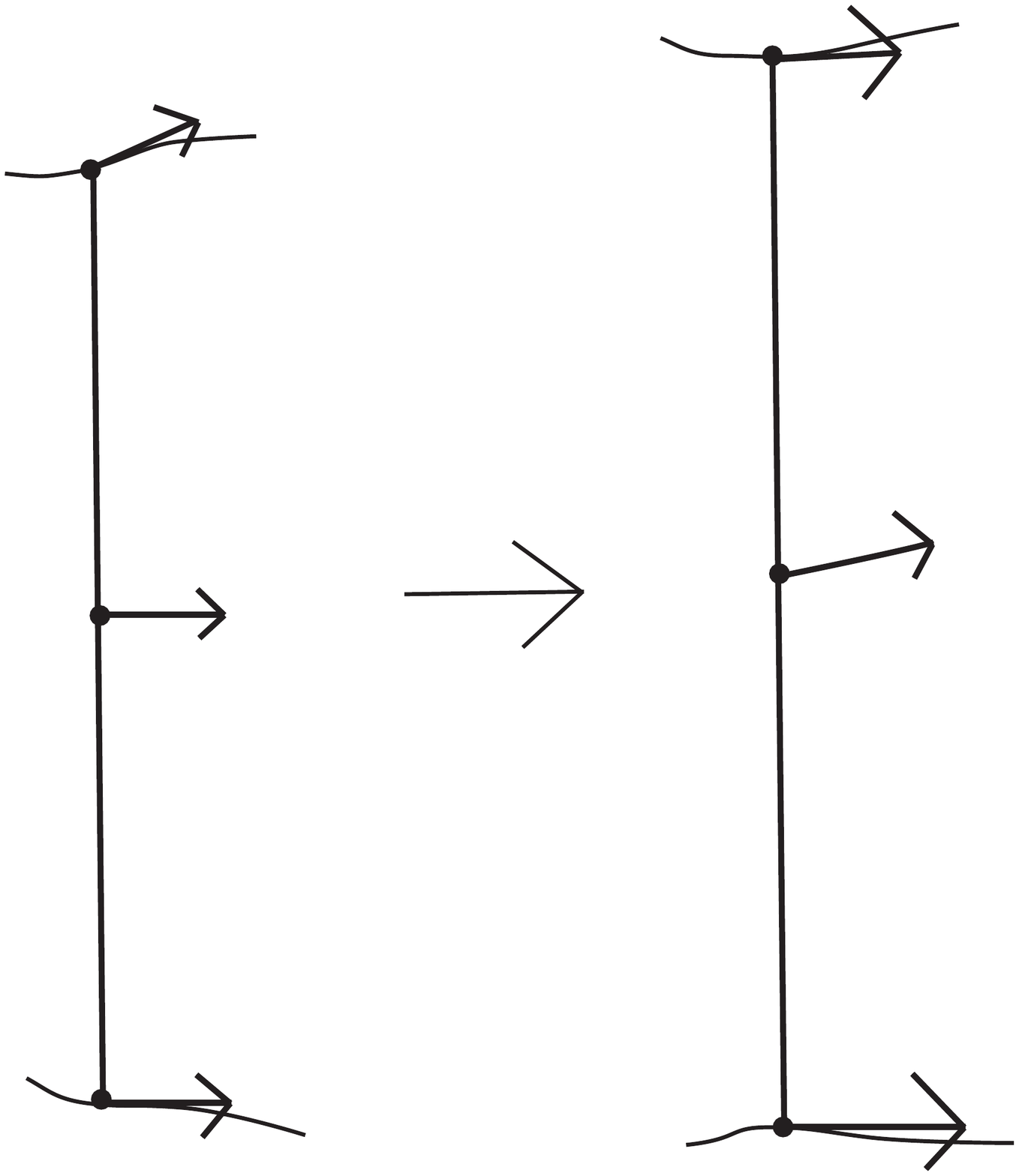} % figure file
 \put(8,34){$\ell$}
\put(15, 50){$v$}
\put(35, 53){$\xi$}
\put(60, 54){$d \xi (v)$}
\put(10, 80){$u$}
\put(10, 8){$w$}
\put(50, 86){$d\xi(u)$}
\put(50, 4){$d\xi(w)$}
%\put(2, 60){$R$}
%\put(45, 64){$R'$}
          %   \put( , ){}  
      \end{overpic}

\caption{}\label{fig:horizontal}

\end{figure}
Let $v$ be a horizontal unit tangent vector at a point $p$ in $R$. 
Let $\ell$ be the vertical leaf of $R$ passing $p$.
Let $u$ and $w$ be the unit tangent vectors tangent to the horizontal edges of $R$ at the endpoints of $\ell$, such that $u, w$ point  the same direction as $v$ in the horizontal coordinate. 

Then if $\del > 0$ is sufficiently small, since $R$ is $(\del, K)$-circular,  $u$ and  $w$ are $\ep$-{\it nearly horizontal}, i.e. the angles between those vectors and appropriate horizontal vectors are less than $\ep$. 
Since $R'$ is $\del$-nearly circular, we can in addition assume that $d \xi (u)$ and $d \xi(w)$ are $\ep$-nearly horizontal.
Notice that $d \xi | \ell$ induces a single affine map from the plane tangent to $\ell$ to the plane tangent to $\xi(\ell)$.
 Thus $d \xi(v)$ is $\ep$-nearly horizontal as well. 
Therefore, if $\del > 0$ is sufficiently small, $1-\ep < |d \xi (v)|  < 1 + \ep$ for all horizontal unit tangent vectors $v$ at points in $R$, since $\xi_\ast$ is a smooth bilipschitz map with small distortion. 
In addition, since $\xi$ is,  in the vertical direction,  a smooth bilipschitz map with small distortion,  $d \xi$ is in particular $(1 + \ep)$-quasiconformal (\textit{i.e.} distorts angles by a factor of at most $1 + \ep$).  

If $\del > 0$ is sufficiently small,  since $\xi$ is a bilipschitz map  of small distortion  in both vertical and horizontal directions, $\xi$ is $(\ep + 1)$-bilipschitz map.  
We have seen that $\xi$ is $K$-rough isometric in the vertical direction and $\ep$-isometric in the horizontal direction. 
Thus if $\del > 0$ is sufficiently small, then $\xi$ is a $K$-rough isometry.  
 \Qed{lem:CompareBranches}

\subsubsection{Linear maps between traintracks}\Label{sMapOnBranches}

Recall that we have traintracks $\TT_N$ on $C_N$ and $\TT_t$ on $C_t$ that are isomorphic as smooth fat traintracks, and their branches are supported on round cylinders.  
Then if $\RR_i$ and $\RRR_i$ are corresponding branches of $\TT_N$ and $\TTT_t$, then let  $\xi_i\col\RR_i \to \RRR_i$ be the diffeomorphism defined in \S \ref{defRectangleMap}.
Then let $\xi  = \{\xi_i \}_i$. 
Although  $\xi_i$'s typically do not match up along vertical edges for adjacent branches of the traintracks (see \S \ref{s:ModityEuclidean}),  
Nonetheless we may regard $\xi$ as a map $\TT_N \to \TTT_t$ that is continuous except along vertical edges.

We check the hypotheses in Lemma \ref{lem:CompareBranches} to $\xi_i$.
Let $R$ be a branch of $T$ on the hyperbolic surface $\tau$.
Let $\RR_N$ and $\RRR_t$ be the  branches of $\TT_N$ and $\TT_t$ corresponding to $R$.
Let $\AA_N$ and $\AA_t$ be the round cylinders supporting  $\RR_N$ and $\RRR_t$, respectively.  \\

Then, there is $K > 0$, such that,  for every $\del > 0$, if $T$ is sufficiently straight, then $\TT_N$ and $\TTT_t$ are $(\del, K)$-nearly circular (Theorem \ref{ThmTraintrack} (2)), which is Assumption (1) in Lemma \ref{lem:CompareBranches}. 
In addition, by Theorem \ref{ThmTraintrack} (1 -a),  the lengths of the cores of $\AA_N$ and $\AA_t$ are $\ep$-close, regardless of the choice of $R$, which is (2) in Lemma \ref{lem:CompareBranches}.

For every $\del  > 0$, if $T$ is sufficiently straight, then by Theorem \ref{ThmTraintrack} (1 -b) and Proposition \ref{p:TN}, 
\begin{itemize}
\item each vertical leaf of each $\RR_N$ has length $\del $-close to the weight of $L _ N$ on $\RR_N$, and
\item each vertical leaf of each $\RRR_t$ has length $\del $-close to the weight of $t N$ on $\RRR_t$.\\
\end{itemize}

Recall that $N = N_t$ is chosen so that the differences of the weights of $L + N$ and $M$ on corresponding branches  are uniformly bounded by the constant $D > 0$ independent on $t$. 

In addition, if $t > 0$ is sufficiently large, weights of $t M$ and $L + N$ are at least $1/\del$ on all branches of the traintracks.
Thus  Lemma \ref{lem:CompareBranches} (\ref{i:height}) holds.
Therefore, by Lemma \ref{lem:CompareBranches}, 

\begin{proposition}\
There is $H > 0$  such that, given any $\ep > 0$, if $T$  is sufficiently straight, then for all sufficiently large $t > 0$, 
$\xi_i\col \RR_i \to \RRR_i$ is a smooth  $(1 + \ep, H)$-bilipschitz map  for each branch $R_i$ of $T$ with respect to the Euclidean metrics $E_N, E_t$ induced by their supporting round cylinders.  
\end{proposition}

\subsubsection{Matching along adjacent branches}\Label{s:ModityEuclidean}
  The map $\xi\col \TT_N \to \TTT_t$ is  defined between corresponding branches (\S \ref{sMapOnBranches}), if $T$ is sufficiently straight,  it is an $(1 + \ep, K)$-bilipschitz map.
 Typically, on adjacent branches of $\TT_N$,  $\xi$ do not agree along on their common vertical edge:  Their supporting round cylinders may not be concentric, and the canonical euclidean metrics on the round cylinders do no agree along their common boundary circle.
Nonetheless those round cylinders are nearly concentric, and therefore $\xi$ is ``close" to being continuous along vertical edges. 
In this section, we modify $\xi$ so that it is in addition continuous and yet  $(1+\epsilon,K)$-bilipschitz. \\

More generally let $\AA_1$ and $\AA_2$ be adjacent round cylinders on $\RS$, which share a boundary circle $\ell$. 
First we make them concentric  by applying a natural M\"{o}bius transformation to  $\AA_2$, keeping them adjacent:
Let  $c_1$ and $c_2$ be the cores of $\AA_1$ and $\AA_2$, respectively.
 Let $P$ be the hyperbolic plane in $\H^3$ bounded by $\ell$, and let $p_1$ and  $p_2$  be the end points of $c_1$ and $c_2$ on $P$.
 Let $\gam \in \PSL$ be such that $\gam$ restricts to a unique translation of $P$ with $\gam(p_2) = p_1$  along the geodesic $g$ connecting $p_1$ and $p_2$.
 Then $\AA_1$ and $\gam \AA_2$ are still adjacent along $\ell$, and $\AA_1$ and $\gam \AA_2$ are indeed concentric. 
Note that $\gam$, being a M\"{o}bius transformation, preserves the canonical Euclidean metric on $\AA_2$.\\

Let $\AA_2' = \gam \AA_2 \sub \RS$. 
Let $c_2' = \gam c_2 \sub \H^3$, the core of $\AA_2'$.
Next we shall modify $\gam$ by post-composing with a piecewise differentiable homeomorphism  $\eta \col \AA_2' \to \AA_2'$ such that $\eta \ci \gam$ restricts to the identity map on $\ell$.
Let $w$ be the length of $c'_2$.
Then parametrize $c_2'$ by arclength as  $c_2' \col [-\frac{w}{2} , \frac{w}{2}] \to \H^3$ so that  the endpoint $c_2^\prime(w/2)$ lies in  $P$.\\

Let $\Conv \AA_2'$ denote the convex hull of $\AA_2'$ in $\H^3$, which is homeomorphic to a solid cylinder. 
For each $t \in [-w/2, w/2]$,  let $P_t$ be the hyperbolic plane  orthogonal to $c_2'$ at  $c_2'(t)$.
Then  $\{P_t\}$  foliates $\Conv \AA_2'$.
Assigning an orientation on the geodesic $g$ in the direction from $p_1$ to $p_2$, 
let $g_t$ be the oriented geodesic on $P_t$ obtained by parallel transport of $g$ along $c_2'$.
Let $\eta_t $ be the translation of the hyperbolic plane $P_t$ along $g_t$ by $d_P (p_2, p_1) \cdot \frac{2t}{w}$ for $t \in [0, \frac{w}{2} ]$ and $\eta_t$ be the identity map for $t \in [-\frac{w}{2}, 0 ]$.
This one-parameter family of sometimes $\eta_t$ yields a piecewise differentiable homeomorphism $\eta \col \Conv \AA_2' \to \Conv \AA_2'$.  
Then $\eta$ restricts a piecewise differentiable homeomorphism  $\eta\col\AA_2^\prime \to \AA_2^\prime$. 
Indeed the composition $\eta \ci \gam\col \AA_2 \to \AA_2'$ fixes the boundary circle $\ell$ pointwise and $\eta \ci \gam (\AA_2) = \AA_2'$, which is concentric to $\AA_1$.

If a projective structure on  a rectangle $R$ is supported on $\AA_2$, then by post-composing  $dev(R)$ with $\eta \ci \gam$, we obtain a new projective structure on $R$. 
Then accordingly we obtained a new Euclidean metric on $R$.

Next we see that this new Euclidean metric is very close to the original Euclidean metric if round cylinders are nearly concentric. 
First, with respect to the canonical Euclidean metric on $A_2'$, we have
\begin{lemma}\label{alpha}
For every $\ep > 0$,  there is $\del > 0$ such that if the length $w$ of the core of $\AA_2'$  is more than $\ep$ and the  endpoints $p_1$ and $p_{2}$ are $\del$-close,  then $\eta\col \AA_2' \to \AA_2'$ is smoothly $(1 + \ep, \ep)$-bilipschitz  and $\ep$-close to the identity map as a piecewise smooth map.
\end{lemma}
\begin{proof}
Let 
$$\displaystyle X = \bigcup_{0 \leq t \leq w/2} \bd P_t ,$$
be the half of the cylinder $\AA'_2$ where $\eta$ is supported. 
Clearly it suffices to show the lemma on $X$. 
 
If  $d_P(p_1, p_2) / w$ is sufficiently small, then for each $t \in [0, w/2]$, $\eta_t$ is smoothly $\ep$-close to the identity map and in particular $(1 + \ep, \ep)$-bilipschitz. 
Therefore, since $\eta$ preserves $P_t$'s,  if $\del > 0$ is sufficiently small, then $\eta$ is $\ep$-rough isometry of $X$. 

Next we show that $\eta$ is $(1 + \ep)$-bilipschitz.
We saw that $d \eta$ is $(1 + \ep)$-bilipschitz in  $\bd P_t$ for each $t$.
Notice that $P_t$ are naturally identified. 
Then, if $\del > 0$ is sufficiently small,  then, for each $t \in [0, \pi/2]$ and $x \in \bd P_t$, we have  $\frac{d \eta_t}{d t}(x) < \ep$.
Thus, since $\eta$ preserves $t$-coordinates, every horizontal tangent vector on $X$ stays $\ep$-almost horizontal and, in terms of its length,  $(1+ \ep)$-bilipschitz.
Therefore, $d \eta$ is bilipschitz in both $t$ and $\bd P_t$ directions and thus $(1 + \ep)$-quasiconformal.
 Hence, for sufficiently small $\del > 0$, $\eta$ is  $(1 + \ep)$-bilipschitz. 
\end{proof}

Given a traintrack  $\TT$ supported on round cylinders, we can apply the modification above for all adjacent branches, so that the supporting cylinders are all concentric. 
Then all supporting cylinders are contained in $\RS$ minus two points, the end points of their common axis. 
Then, letting $\til{\TT}$ be the universal cover of $\TT$, there is a modified developing map from $\til{\TT}$ to $\RS$ minus two points that is equivariant under a representation of $\pi_1(\TT)$ onto a subgroup of $\PSL$ consisting of loxodromic elements fixing the common axis. 
Thus we have a continuous Euclidean metric on $\TT$ induced by this modified developing map ({\it modified Euclidean metric}). 
Then Lemma \ref{alpha} implies
\begin{proposition}\Label{pModifyT}
For every $\ep > 0$, if $\del > 0$ is sufficiently small, then 
given any $(\del, \ep)$-nearly circular traintrack supported on round cylinders, the modified Euclidean metric is $(1 + \ep, \ep)$-bilipschitz to the original Euclidean metric on the traintrack on each branch.  
\end{proposition}

Now we apply to this modification to our setting. 
Let $E_N'$ be the continuous Euclidean metric on $\TT_N$ obtained by making supporting cylinders concentric as above. 
Similarly let $E_t'$ be the continuous Euclidean metric on $\TTT_t$.
Then, by Proposition \ref{pModifyT}, for every $\ep > 0$, if $T$ is sufficiently straight, then $E_N'$ and $E_t'$ are piecewise smoothly $(1 + \ep, \ep)$-bilipschitz to $E_N$ and $E_t$, respectively. 
Therefore,  for every $\ep > 0$, if $T$ is sufficiently straight, then $(\TT_N, E_N')$ and $(\TTT_t, E')$ are also an $(\ep, K)$-nearly circular traintrack (for all $t > 0$). 
Then, since $E_N'$ and $E_t'$ are continuous also in the intersection of vertical edges of adjacent branches,  by Lemma \ref{lem:CompareBranches} and Lemma \ref{alpha} we have, 
\begin{proposition}\Label{p:phi}
There is $H > 0$ such that, for every $\ep > 0$,  if $T$ is sufficiently straight, then for sufficiently large $t > 0$,  the  map $\phi\col (\TT_N, E_N') \to (\TTT_t, E_t')$ given by \ref{defRectangleMap}, is a piecewise-differentiable homeomorphism on $|\TT_N|$ and $(1 + \ep, H)$-bilipschitz on each branch; 
moreover $\phi$ is $\ep$-close to $\xi\col (\TT_N, E_N) \to (\TTT_t, E_t)$ piecewise smoothly on each branch. 
\end{proposition}
Clearly, in this proposition, we may in addition assume that $\phi$ is $(1 + \ep)$-quasiconformal.

\subsection{Step 3: Gluing the maps between traintracks and their complements}

Notice that $\bd \TT_N$ enjoys different metrics induced from the Euclidean metric $E_N'$ from $\TT_N$ and from the hyperbolic metric in the complement $C_N \minus \TT_N$. 
Similarly $\bd \TTT_t$ enjoys different metrics from  the Euclidean metric $E_t'$ on $\TTT_t$ and from the hyperbolic metric in $C_t \minus \TTT_t$.
\begin{proposition}\label{lem:bdry}
For ever $\ep > 0$, if $T$ is sufficiently straight, then 
$\bdr \phi  \col \bdr \TT_N \to \bdr \TTT_t$ is $(1 + \ep, \ep)$-bilipschitz with respect to the hyperbolic metric in the complement of the traintracks.
\end{proposition}

\begin{proof}
By Proposition \ref{p:phi},
$\xi$ and $\phi$ are piecewise smoothly $\ep$-close. 
Thus it suffices to show $\bd \xi \col \bd\TT_N  \to \bd \TTT_t$ is $(1 + \ep, \ep)$-bilipschitz map. 

We first show that $\bd \xi \col \bd \TT_N \to \bd\TTT_t$ is $(1 + \ep)$-bilipschitz with respect to the complementary hyperbolic metrics. 
Let $\RR$ be a branch of the universal cover $\til{\TT}_N$ supported on a round cylinder $\AA$.
Let $\zeta_N\col \AA \to c$ be the projection into the core $c$ of $\AA$. 
Let $H_N$ denote the Euclidean distance in the horizontal direction on a branch $\RR$, so that  $\zeta_N$ is isometric with $H_N$ on $\AA$ and the induced hyperbolic metric on $c$.
Then, if $T$ is sufficiently straight, since  each horizontal edge $e$ of $\RR$ is nearly orthogonal to the vertical foliation on $\AA$,  $H_N$ on $C$ is $(1 + \ep, \ep)$-bilipschitz to the metric induced from the Euclidean metric $E_N$ on $\RR$. 

Let $\beta_N\col \H^2 \to \H^3$ be the pleated surface for $C_N$ equivariant under $\hol (C_N)$.
Recall that  $\kap_N\col C_N \to \tau$ is the composition of  the collapsing maps  $C_N \to C$ and $C \to \tau$. 
Let $\til{\kap}_N\col \til{C} \to \h^2$ be the lift to a map between their universal covers.
Then the restriction of  $\beta_N \ci  \til{\kap}_N\col \til{C}_N \to \H^3$ to $e$ is a smooth $(1 + \ep)$-bilipschitz curve, and  it is smoothly $\ep$-close to the core of $\AA$ (Proposition \ref{CloseToCore}). 
Let  $\mathcal{Y}$  be the stratum of $(\til{C}_N, \til{\LL})$ containing $e$. 
Then $\til{\kap}_N$ takes $\mathcal{Y}$ to a stratum $Y$ of  $(\H^2, \til{L})$.
Note that $dev(C)$ embeds $\mathcal{Y}$ into $\RS$.
Then $\beta(Y)$ is contained in a hyperbolic plane in $\h^3$ so that there the nearest point projection onto the hyperbolic plane inducing an isometry $\eta\col \mathcal{Y} \to \beta(Y)$ with respect to the Thurston metric on  $\mathcal{Y}$.  
In particular $\eta| e$ is  an isometric embedding. 
Therefore, since $\beta_N \ci \til{\kap}_N (e)$ is sufficiently close to the core of $\AA$, if $T$ is sufficiently straight, then, for each horizontal edge $e$ of $\TT_N$,   we see that the identity map from $e$ with $H_N$ to $e$ with the complementary hyperbolic metric is $(1 + \ep)$-bilipschitz.

Similarly, if $T$ is sufficiently straight,  for each horizontal edge of $\TTT_t$, the horizontal  Euclidean distance given by $E_t$ is $(1 + \ep)$-bilipschitz to the complementary hyperbolic metric.

\begin{figure}
\begin{overpic}[scale=.3%, grid,tics=10
] {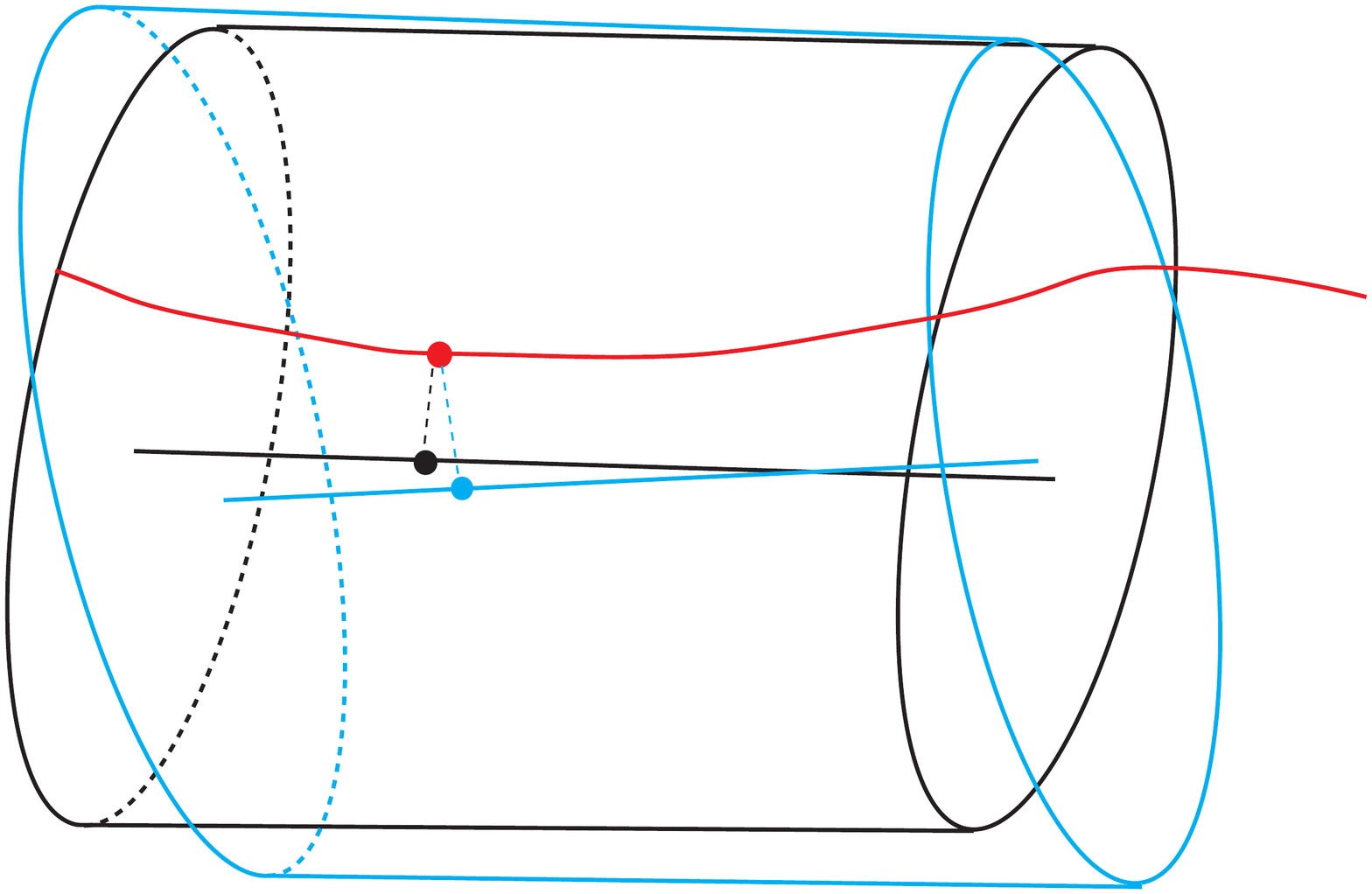} % figure file
\put(40, 50){$\textcolor{red}{p}$}
\put(50, 50){$\textcolor{red}{\bd \TT_N}$}
\put(50, 35){$c_N$}
\put(35, 35){$\textcolor{Cyan}{c_t}$}
          %   \put( , ){}  
      \end{overpic}
\caption{}\label{}
\end{figure}

Recall that $\xi_i \col \RR_i \to \RRR_i$ induces a linear map $\xi_{i\, \ast}$ in their horizontal coordinates, and, for every $\ep > 0$, if $T$ is sufficiently straight, all $\xi_{i\,\ast}$ is $(1 + \ep, \ep)$-bilipschitz with the horizontal distances given by $E_N$ and $E_t$.  
Since we have seen the horizontal Euclidean distances are bilipschitz to the complementary hyperbolic metrics, hence, if $T$ is sufficiently straight,  for each $\xi_i \col \RR_i \to \RRR_i$ and each horizontal edge $e$ of $\RR_i$, we see that $\xi_i | e$ is $(1 + \ep)$-bilipschitz with the complementary hyperbolic metrics.

Next we show that $\bdr \xi \col \bdr \TT_N \to \bdr \TTT_t$ is an $\ep$-rough isometry on each horizontal edge (of each branch). 
The collapsing maps $\kap_N \col C_N \to \tau$ and $\kap_t \col C_t \to \tau$ respectively take 
$\TT_N $ and $\TTT_t $ to a single nearly straight traintrack $T$ on $\tau$ up to small homotopies of vertical edges.
 Suppose that $e_N$ and $e_t$ are corresponding horizontal edges  of $\TT_N$ and $\TTT_t$ that descends to a common edge $e$ of $T$.
Then  $\kap_N\col C_N \to \tau$ and $\kap_t\col C_t \to \tau$ takes $e_N$ and $e_t$ to $e$ up to $\ep$-small isotopy  (Theorem \ref{ThmTraintrack}). 
Recall that $\bdr \TT_N \cong \bdr \TTT_t$  as projective structures. 
Then with respect to this identification, if $T$ is sufficiently straight, then $e_N$ and $e_t$ intersects in the curve $e_N \cap e_t  =: h$ and the endpoints of $h$ are $\ep$ close to the endpoints of $e_N$ and $e_t$. 
 
Let $\AA_N$ and $\AAA_t$ be the round cylinders supporting $e_N$ and $e_t$, so that  $e_N \sub \AA_N$ and $e_t \sub \AAA_t$. 
Thus place the supporting cylinders $\AA_N$ and $\AAA_t$ on $\RS$ so that it realizes the  intersection of  $e_N$ and $e_t$ in $h$.  
In addition, since $\TT_N$ and $\TTT_t$ are $(\ep, K)$-nearly circular, $e_N$ is contained in a circular rectangle supported on $\AA_N$ of height less than $\ep$, and $e_t$ is contained in a circular rectangle of height less than $\ep$ that is supported on $\AAA_t$.
Therefore, if $T$ is sufficiently straight,  the cores of  $\AA_N$ and $\AAA_t$  are $\ep$-close in $\H^3$.

Since $h$ is $\ep$-close to $e_N$ and $e_t$,  it suffuses to show $\bdr \xi$ is $\ep$-rough isometry  on $h$. 
Let $c_t \sub \H^3$ be the  core of $\AAA_t$, and let $\zeta_t\col \AAA_t \to c_t$ be the projection along geodesics in $\H^3$.
Similarly let $c_N \sub \H^3$ be the core of $\AA_N$, and let $\zeta_N\col \AA_N \to c_N$ be the projection along geodesics. 
Since $c_N$ and $c_t$ are $\ep$-close  in $\H^3$,  if $T$ is sufficiently straight, then for every $p \in h$, $\zeta_t(p)$ and $\zeta_N(p)$ are $\ep$-close in $\H^3$. 
In addition the linear map $\xi_\ast \col c_t \to c_N$ is $\ep$-close to the identity. 
Therefore, since  $\xi_\ast$ is an $\ep$-rough isometry, if $T$ is sufficiently straight,   $\xi_\ast$ takes $\zeta_t(p)$ to a point $\ep$-close to $\zeta_N(p)$ for each $p \in h$.
Since the hyperbolic metrics are $\ep$-rough isometric to the horizontal metrics,  $\xi$ is $\ep$-rough isometric  on $h$ with the complementary hyperbolic metrics. 
\end{proof}

\subsubsection{Matching $\phi$ and $\psi$ along the boundary of the traintrack.} 

Although $\phi\col \TT_N \to \TTT_t$ and $\psi\col C_N \minus \TT \to C_t \minus \TTT$ do not coincide along the boundary of $\TT_N$,
however by Proposition \ref{lem:bdry}, for every $\ep > 0$, if $T$ is sufficiently straight, then $\psi \circ \phi^{-1}$, defined on $\bd \TTT_t$, is $(1 + \ep, \ep)$-bilipschitz map  (with either metrics). 
Thus,  for sufficiently large $t > 0$, we  modify $\phi$ so that $\phi$ and $\psi$ agree along $\bd \TT$ by post-composing $\phi$ with $(1 + \ep, \ep)$-bilipschitz map.

Let $\NNN(= \NNN_t)$ be the {\it  unit collar} of $\bdr \TTT_t$ in $\TTT_t$. 
That is, the union of  the $1$-neighborhoods of the horizontal edges in the branches of $\TTT_t$ (Figure \ref{Hcollar}) with respect to the modified euclidean metric $E'_t$ included by concentric cylinders on $\RS$ (\S\ref{s:ModityEuclidean}). \\

Recall that if $t > 0$ is sufficiently large, then  all branches of $\TTT_t$ has height more than $2$. 
Then $\NNN$ is a union of finitely many Euclidean cylinders and rectangles of height $1$ that are disjoint except at vertices on the switch points on $\TTT_t$.
\newcommand{\XXX}{\mathsf{X}}
Thus set $\NNN = \cup_j \XXX_j$, where $\XXX_j$ are the Euclidean cylinders and rectangles.

\begin{figure}[H]
\begin{overpic}[scale=.25%, grid,tics=10
] {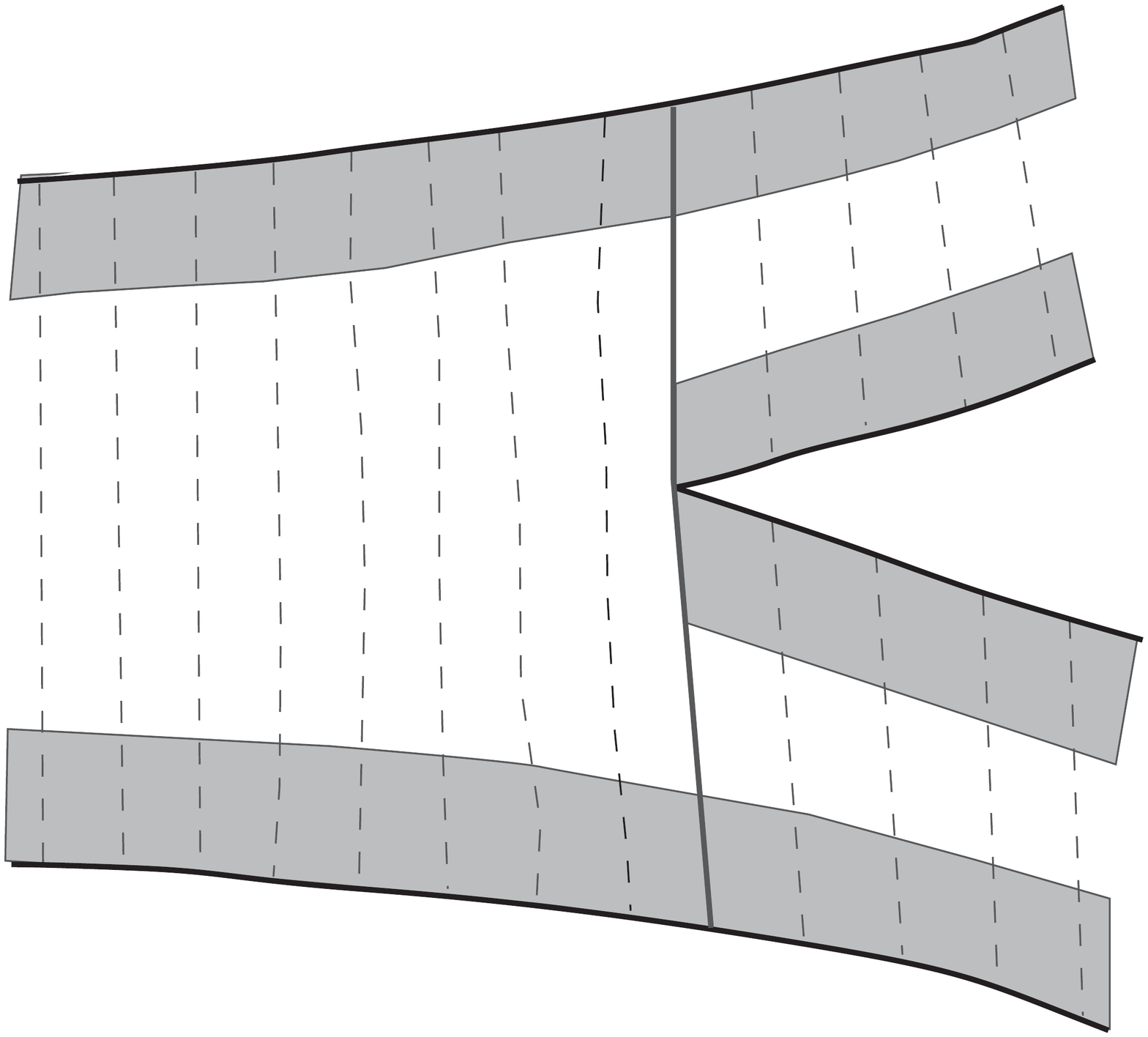} % figure file
          %   \put( , ){}  
      \end{overpic}
\caption{}\label{Hcollar}
\end{figure}

If $\XXX_j$ is a rectangle, let $a_j$ be its horizontal edge contained in in the interior of $\TTT_t$ and $b_j$ be its horizontal edge in the boundary of $\TTT_t$.  
Similarly, if $\XXX_j$ is a cylinder, let $a_j$ be  its boundary circle contained in the interior of $\TTT_t$ and $b_j$ be the other boundary circle contained in the boundary of $\TTT_t$.
Note that  $\psi \circ \phi^{-1} \col \TTT_t \to \TTT_t$ preserves each $b_j$.

Since  $\psi \circ \phi^{-1}$ on $\bd \TTT_t$ is a $(1 + \ep, \ep)$-bilipschitz map, 
 we can construct a desired bilipschitz map  supported on $\NNN$:
The vertical foliation of $\TTT_t$ induces a vertical foliation on each $\XXX_j$, each leaf of which connects a point on $a_j$ and a point on  $b_j$.
Then there is a  unique homeomorphism $\eta_j\col \XXX_j \to \XXX_j$ such that

\begin{itemize}
\item $\eta_j$ is the identity map on $a_j$,
\item $\eta_j$ is $\psi \circ \phi^{-1}$ on $b_j$, and
\item $\eta_j$ is linear on each vertical leaf. 
\end{itemize}
Since, $\psi \circ \phi^{-1}$ on $b_j$ is a piecewise differentiable homeomorphism,  so is $\eta_j$.
Then, for every $\ep > 0$, if $T$ is sufficiently straight,   then for sufficiently large $t > 0$, $\eta_j \col X_j \to X_j$ are $(1+\ep, \ep)$-bilipschitz. 
Let $\eta\col \TT_N \to \TTT_t$ be the homeomorphism such that $\eta = \eta_j$ on each $X_j$ and the identity map on $\TTT_t \minus \NNN$. 
Then $\eta$ is $(1 + \ep, \ep)$-bilipschitz and $\eta \circ \phi = \psi$ on $\bdr \TT_N$. \\

Since $\eta \ci \phi \col \TT_N \to \TTT_t$ and $\psi\col C_N \minus \TT_N \to C_t \minus \TTT_t$ are both $(1 + \ep, \ep)$-bilipschitz,  we obtain a desired $(1 + \ep, \ep)$-bilipschitz map $C_N \to C_t$ for sufficiently large $t > 0$.

\bibliographystyle{alpha}

%\bibliography{Reference}
\bibliography{/Users/shinpeibaba/Dropbox/MyPaper/Reference.bib}

\end{document}